\documentclass[10pt]{amsart}
\usepackage{times,amsmath,amsbsy,amssymb,amscd,mathrsfs}

\usepackage{algorithm2e} 
\usepackage{multicol,multirow}
\usepackage{mathtools}
\usepackage[usenames,dvipsnames,svgnames,table]{xcolor}
\usepackage[all]{xy}
\usepackage{wrapfig}
\usepackage{tcolorbox}

\usepackage{tikz,tikz-cd}
\usepackage[utf8]{inputenc}
\usepackage{pgfplots} 
\usepackage{pgfgantt}
\usepackage{pdflscape}
\pgfplotsset{compat=newest} 
\pgfplotsset{plot coordinates/math parser=false}
\newlength\fwidth

\usepackage[numbered]{mcode}

\definecolor{myBlue}{rgb}{0.0,0.0,0.55}
\usepackage[pdftex,colorlinks=true,citecolor=myBlue,linkcolor=myBlue]{hyperref}

\usepackage[hyperpageref]{backref}

\usepackage{comment,enumerate,multicol,xspace}

  \newcounter{mnote}
  \let\oldmarginpar\marginpar
    \renewcommand\marginpar[1]{\-\oldmarginpar[\raggedleft\footnotesize #1]%
    {\raggedright\footnotesize #1}}

\newtheorem{theorem}{Theorem}[section]
\newtheorem{lemma}[theorem]{Lemma}
\newtheorem{corollary}[theorem]{Corollary}

\newtheorem{remark}[theorem]{Remark}

\newcommand{\dd}{\,{\rm d}}

\newcommand{\bs}{\boldsymbol}

\newcommand{\mc}{\mcode}

\DeclareMathOperator*{\spa}{span}

\newcommand{\curl}{{\rm curl\,}}
\renewcommand{\div}{\operatorname{div}}
\newcommand{\grad}{{\rm grad\,}}
\DeclareMathOperator*{\rot}{rot}

\newcommand{\tr}{\operatorname{tr}}

\newcommand{\skw}{\operatorname{skw}}

\newcommand{\mskw}{\operatorname{mskw}}

\newcommand{\vertiii}[1]{{\left\vert\kern-0.25ex\left\vert\kern-0.25ex\left\vert #1 
    \right\vert\kern-0.25ex\right\vert\kern-0.25ex\right\vert}}

\usepackage{stmaryrd}
\usepackage{scalefnt}
\usepackage{graphicx}
\usepackage{listings}
\newcommand{\Oplus}{\ensuremath{\vcenter{\hbox{\scalebox{1.5}{$\oplus$}}}}}

\begin{document}
\title[High Order Edge and Face Elements]{Geometric Decomposition and Efficient Implementation of High Order Face and Edge Elements}
\author{Chunyu Chen}%
\address{School of Mathematics and Computational Science, Xiangtan University; National Center of  Applied Mathematics in Hunan, Hunan Key Laboratory for Computation and Simulation in Science and Engineering Xiangtan 411105, China }%
 \email{202131510114@smail.xtu.edu.cn}%
\author{Long Chen}%
 \address{Department of Mathematics, University of California at Irvine, Irvine, CA 92697, USA}%
 \email{chenlong@math.uci.edu}%
 \author{Xuehai Huang}%
 \address{School of Mathematics, Shanghai University of Finance and Economics, Shanghai 200433, China}%
 \email{huang.xuehai@sufe.edu.cn}%
\author{Huayi Wei}%
\address{School of Mathematics and Computational Science, Xiangtan University; National Center of  Applied Mathematics in Hunan, Hunan Key Laboratory for Computation and Simulation in Science and Engineering Xiangtan 411105, China }%
 \email{weihuayi@xtu.edu.cn}%

 \thanks{The first and fourth authors were supported by the National Natural Science
Foundation of China (NSFC) (Grant No. 12371410, 12261131501), and the
construction of innovative provinces in Hunan Province (Grant No. 2021GK1010).
The second author was supported by NSF DMS-2012465 and DMS-2309785. The third
author was supported by the National Natural Science Foundation of China (NSFC)
(Grant No. 12171300), and the Natural Science Foundation of Shanghai (Grant No.
21ZR1480500).}

\makeatletter
\@namedef{subjclassname@2020}{\textup{2020} Mathematics Subject Classification}
\makeatother
\subjclass[2020]{
65N30;   
35Q60;   
}

\begin{abstract}
This study investigates high-order face and edge elements in finite element methods, with a focus on their geometric attributes, indexing management, and practical application. The exposition begins by a geometric decomposition of Lagrange finite elements, setting the foundation for further analysis. The discussion then extends to $H(\div)$-conforming and $H(\curl)$-conforming finite element spaces, adopting variable frames across differing sub-simplices. The imposition of tangential or normal continuity is achieved through the strategic selection of corresponding bases. The paper concludes with a focus on efficient indexing management strategies for degrees of freedom, offering practical guidance to researchers and engineers. It serves as a comprehensive resource that bridges the gap between theory and practice.
\end{abstract}
\keywords{Implementation of Finite Elements, Nodal Finite Elements, 
    $H(\rm{curl})$-Conforming, $H(\rm{div})$-Conforming}

\maketitle

\section{Introduction}
This paper introduces node-based basis functions for high-order finite elements, specifically focusing on Lagrange, BDM (Brezzi-Douglas-Marini)~\cite{BrezziDouglasMarini1985,Nedelec1986,BrezziDouglasDuranFortin1987}, and second-kind Nédélec elements~\cite{Nedelec1986,ArnoldFalkWinther2006}. These elements are subsets of the spaces $H^1$, $H(\div)$, and $H(\curl)$, with their shape functions being the full polynomial space $\mathbb P_k^n$, where $k$ represents the polynomial degree, and $n$ the geometric dimension. Notably, varying continuity across these elements gives rise to distinct characteristics. When $n=3$, for the lowest degree $k=1$, degrees of freedom (DoFs) of the $H(\div)$-conforming finite element are posed on faces and DoFs of the $H(\curl)$-conforming finite element are posed on edges. Therefore, conventionally an $H(\div)$-conforming finite element is referred to as a face element, and an $H(\curl)$-conforming element is an edge element.  
They are also known as the second family of face and edge element as the shape function space is the full polynomial space while the first family consists of incomplete polynomial spaces \cite{ArnoldFalkWinther2006}.

In the realm of Lagrange finite elements, nodal basis functions stand out for
their simplicity and ease of computation. In contrary, constructing basis
functions for face and edge elements is more intricate. Traditional approaches involve the Piola transformations, where basis
functions are first devised on a reference element and subsequently mapped to
the actual element using either covariant (to preserve tangential continuity, in
the case of edge elements) or contravariant (to maintain
normal continuity, for face elements) Piola transformations. Detailed explanations
of this approach can be found in
\cite{ervin2012computational,EfficientAssemblyOfHdivHcurl}, and implementation is in open-source software such as MFEM~\cite{anderson2021mfem} and Fenics~\cite{alnaes2015fenics}.

Arnold, Falk and Winther, in~\cite{ArnoldFalkWinther2009}, introduced a
geometric decomposition of polynomial differential forms.
Basis functions based on Bernstein polynomials were proposed, paving the way for subsequent advancements. In
\cite{ainsworth2015bernstein,ainsworth2018bernstein}, basis functions founded on Bernstein polynomials were explored, accompanied by fast algorithms for the matrix assembly.  Additionally, hierarchical basis functions for $H(\curl)$-conforming
finite elements were introduced in
\cite{ainsworth2003hierarchic,xin2011well,xin2011construction,xin2013construction}.

While these methods offer valuable insights, they can be quite complex.
Researchers have thus ventured into simpler approaches. In
\cite{ChristiansenHuHu2018finite}, a method multiplying scalar nodal finite
element methods by vectors was introduced, resulting in $H(\div)$ and $H(\curl)$
conforming finite elements that exhibit continuity on both vertices and edges.

We propose a straightforward method to construct nodal bases for the second family face and edge elements. Initially, we clarify the basis of Lagrange elements by considering them dual to the degrees of freedom, which are determined by values at interpolation points. We then extend this principle to vector polynomial spaces, wherein each interpolation point establishes a frame that includes both tangential and normal (t-n) directions. We impose the continuity of either tangential or normal components by appropriate choice of the t-n decomposition. We explicitly derive the dual basis functions for these elements from the basis functions of Lagrange elements.


This idea has been previously explored in~\cite{DeSiqueiraDevlooGomes2010,DeSiqueiraDevlooGomes2013,CastroDevlooFariasGomesEtAl2016} for constructing a hierarchical basis of $H(\div)$ elements in two and three dimensions. Through a rotation process, it was also adapted for $H(\curl)$ elements in two dimensions, as detailed in~\cite{DeSiqueiraDevlooGomes2010,DeSiqueiraDevlooGomes2013}. However, extending their methodology to higher dimensions presented significant challenges. Our work advances this field by developing a geometric decomposition of the second family face and edge elements in arbitrary dimensions and orders. Additionally, we introduce degrees of freedom that are dual to the basis functions, aspects not covered in the aforementioned studies.

An essential aspect of practical finite element implementations is the
management of local DoFs, ensuring proper mapping to global
DoFs. This is critical for maintaining the correct continuity between elements
and correctly integrating matrices and vectors across the system. Despite its importance, there is a scarcity of literature specifically addressing DoF management for high-order finite elements. In this paper, we explore the global indexing of interpolation points in Lagrange finite elements of arbitrary degree. Our discussion comprehensively covers DoF management for Lagrange, the second family face and edge elements. Our goal is to facilitate the implementation of high-order finite element methods by simplifying their DoF management.
 
The subsequent sections of this paper are organized as follows. Section
\ref{sec:pre} elucidates foundational concepts such as simplicial lattices,
interpolation points, bubble polynomials, and triangulations. Sections
\ref{sec:lagrange}, \ref{sec:facefem}, and \ref{sec:edgefem} delve into the
construction of basis functions for Lagrange, the second family face and edge elements,
respectively. Section \ref{sec:implementation} addresses the management of
degrees of freedom, a critical aspect of effective matrix assembly. Finally,
Section \ref{sec:numerexamples} presents two numerical examples, solving the
mixed Poisson problem and Maxwell problem using the second family face and edge elements, respectively, to validate the correctness of the proposed basis function construction method.

\section{Preliminaries}\label{sec:pre}
In this section, we provide essential foundations for our study. We introduce multi-indices, simplicial lattices, and interpolation points. We explain sub-simplices and their relations. Furthermore, we introduce the dictionary ordering of simplicial lattices and the concept of bubble polynomials. 

\subsection{Simplicial lattice}
A multi-index of length $n+1$ is an array of non-negative integers: 
$$
\boldsymbol\alpha = (\alpha_0, \alpha_1, \ldots, \alpha_{n}), \quad \alpha_i \in \mathbb N, i=0,\ldots, n.
$$ 
The degree or sum of the multi-index is 
$|\boldsymbol \alpha| = \sum_{i = 0}^{n} \alpha_i$ and factorial is 
$\boldsymbol \alpha! = \prod_{i=0}^n (\alpha_i!)$. 
The set of all multi-indices of length $n+1$ and degree 
$k$ will be called a simplicial lattice and denoted by $\mathbb T^n_k$, i.e.,
$$
\mathbb T^n_k = \{ \boldsymbol \alpha \in \mathbb N^{n+1}:
|\boldsymbol \alpha| = k\}.
$$ 
The elements in $\mathbb T^n_k$ can be linearly indexed by the 
dictionary ordering $R_n$:
\begin{equation*}
R_n(\boldsymbol \alpha)=\sum_{i=1}^n{\alpha_i+\alpha_{i+1}+\cdots+\alpha_n+n-i \choose n+1-i}.
\end{equation*}
For example, for an element $\boldsymbol \alpha$ in $\mathbb T^2_k$, the index is given by the mapping:
$$
R_2(\boldsymbol \alpha) = \frac{(\alpha_1+\alpha_2)(\alpha_1+\alpha_2+1)}{2} + 
\alpha_2.
$$ 
Notice that $\alpha_0$ is not used in the calculation of $R_n(\bs \alpha)$. 
 
\subsection{Interpolation points}
Let $\boldsymbol x_0, \boldsymbol x_1, \ldots, \boldsymbol x_n$ be $n+1$ points in $\mathbb R^n$ and 
$$T = {\rm Convex}(\boldsymbol x_0, \boldsymbol x_1, \ldots, \boldsymbol x_n) =\left \{ \sum_{i=0}^n \lambda_i \bs x_i : 0\leq \lambda_i \leq 1, \sum_{i=0}^n\lambda_i = 1 \right \},$$ 
be an $n$-simplex, where $\bs \lambda=(\lambda_0, \ldots, \lambda_n)$ is called the barycentric coordinate. We can have a geometric embedding of the algebraic set $\mathbb T^n_k$ as follows:
$$
\mathcal X_{T} = \left \{\boldsymbol
x_{\boldsymbol \alpha} = \frac{1}{k}\sum_{i = 0}^n \alpha_i \boldsymbol
x_i: \boldsymbol \alpha\in \mathbb T^n_k \right \},
$$
which is called the set of interpolation points with degree $k$ on $T$; see Fig.\ref{fig:interpointk4} for $k=4$ and $n=2$. In literature~\cite{nicolaides1972class}, it is called the $k$-th order principal lattice of $n$-simplex $T$. Given $\bs \alpha\in \mathbb T^{n}_k$, the barycentric coordinate of $\bs \alpha$ is given by
$$
\bs \lambda(\bs \alpha) = (\alpha_0, \alpha_1, \ldots, \alpha_n )/k.
$$ The ordering of $\mathcal X_T$ is also given by $R_n(\bs\alpha)$.  Note
that the indexing map $R_n(\bs\alpha)$ is only a local ordering of the
interpolation points on one $n$-simplex. In Section \ref{sec:implementation}, we will discuss the global
indexing of all interpolation points on the triangulation composed of simplexes.

\begin{figure}[htp]
\centering
\includegraphics[width=0.5\textwidth]{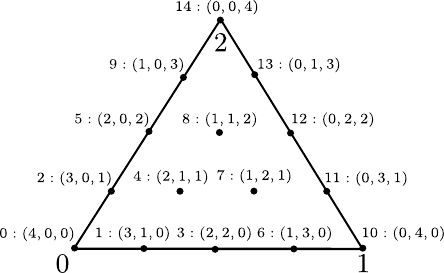}
\caption{The interpolation points of $\mathbb T^2_4$ and their
linear indices and multi-indexs. }
\label{fig:interpointk4}
\end{figure}

By this geometric embedding, we can apply operators to the geometric simplex $T$. For example, $\mathcal X_{\mathring{T}}$ denotes the set of interpolation points in the interior of $T$ and $\mathcal X_{\partial T}$ is the set of interpolation points on the boundary of $T$.

\subsection{Sub-simplexes and sub-simplicial lattices}
Following~\cite{ArnoldFalkWinther2009}, let $\Delta(T)$ denote all the subsimplices of $T$, while $\Delta_{\ell}(T)$ denotes the set of subsimplices of dimension $\ell$ for $\ell=0,\ldots, n$. 
The capital letter $F$ is reserved for an $(n-1)$-dimensional face of $T$ and $F_i\in\Delta_{n-1}(T)$ denotes the face opposite to ${\bs x}_i$ for $i=0,1,\ldots, n$.

For a sub-simplex $f\in \Delta_{\ell}(T)$, following~\cite{Chen;Huang:2022FEMcomplex3D}, we will overload the notation $f$ for both the geometric simplex and the algebraic set of indices. Namely $f = \{f(0), \ldots, f(\ell)\}\subseteq \{0, 1, \ldots, n\}$ and 
$$
f ={\rm Convex}({\bs x}_{f(0)}, \ldots, {\bs x}_{f(\ell)}) \in \Delta_{\ell}(T)
$$
is the $\ell$-dimensional simplex spanned by the vertices ${\bs x}_{f(0)}, \ldots, {\bs x}_{f( \ell)}$.

If $f \in \Delta_{\ell}(T)$, then $f^{*} \in \Delta_{n- \ell-1}(T)$ denotes the sub-simplex of $T$ opposite to $f$. When treating $f$ as a subset of $\{0, 1, \ldots, n\}$, $f^*\subseteq \{0,1, \ldots, n\}$ so that $f\cup f^* = \{0, 1, \ldots, n\}$, i.e., $f^*$ is the complement of set $f$. Geometrically,
$$
f^* ={\rm Convex}({\bs x}_{f^*(1)}, \ldots, {\bs x}_{f^*(n-\ell)}) \in \Delta_{n- \ell-1}(T)
$$
is the $(n- \ell-1)$-dimensional simplex spanned by vertices not contained in $f$.

Given a sub-simplex $f\in \Delta_{\ell}(T)$, through the geometric embedding $f \hookrightarrow T$, we define the prolongation/extension operator $E: \mathbb T^{\ell}_k \to \mathbb T^{n}_k$ as follows:
\begin{equation}\label{eq:Ealpha}
E(\alpha)_{f(i)} = \alpha_{i}, i=0,\ldots, \ell, \quad \text{ and } E(\alpha)_j = 0, j\not\in f.
\end{equation}
Take $f = \{ 1, 3, 4\}$ for example,
then the extension $ E(\alpha)= (0, \alpha_{0}, 0, \alpha_{1}, \alpha_{2}, \ldots, 0)$ for 
$\alpha = ( \alpha_{0},\alpha_{1},\alpha_{2})\in \mathbb T^{\ell}_k(f)$. 
With a slight abuse of notation, for a node $\alpha_f\in \mathbb T^{\ell}_k(f)$, we still use the same notation $\alpha_f\in \mathbb T^{n}_k(T)$ to denote such extension. 
The geometric embedding $\bs x_{E(\alpha)}\in f$ which justifies the notation $\mathbb T^{\ell}_k(f)$ and its geometric embedding will be denoted by $\mathcal X_f$, which consists of interpolation points on $f$. $\mathbb T^{\ell}_k(\mathring{f})$ is the set of lattice points whose geometric embedding is in the interior of $f$, i.e., $\mathcal X_{\mathring{f}}$.

The Bernstein representation of polynomial of degree $k$ on a simplex $T$ is
$$
\mathbb P_k(T) := \textrm{span} \{ \lambda^{\boldsymbol \alpha} = \lambda_{0}^{\alpha_0}\lambda_{1}^{\alpha_1}\ldots \lambda_{n}^{\alpha_n}, \alpha\in \mathbb T_k^{n}\}.
$$
The bubble polynomial of $f$ is a polynomial of degree $\ell + 1$:
$$
b_f: = \lambda_{f} = \lambda_{f(0)}\lambda_{f(1)}\ldots \lambda_{f(\ell)}\in \mathbb P_{\ell + 1}(f).
$$

%
%
 
\subsection{Triangulation}
Let $\Omega$ be a polyhedral domain in $\mathbb R^n$, $n\geq1$. A geometric
triangulation $\mathcal T_h$ of $\Omega$ is a set of $n$-simplices such that
$$
\cup_{T\in \mathcal T_h} T = \Omega, \quad \mathring T_i \cap \mathring T_j =
\emptyset, \ \forall\ T_i, T_j \in \mathcal T_h, T_i \neq T_j.
$$
The subscript $h$ denotes the diameter of each element and can be understood as a piecewise constant function on $\mathcal T_h$. 
A  triangulation is conforming if the intersection of two simplexes are common lower dimensional sub-simplex. We shall restrict to conforming triangulations in this paper. 

The interpolation points on a conforming triangulation $\mathcal T_h$ is 
\begin{equation}\label{eq:Xunion}
\mathcal X = \bigcup_{T\in \mathcal T_h}\mathcal X_{T}.
\end{equation}
Note that a lot of duplications exist in \eqref{eq:Xunion}. A direct sum of the interpolation set is given by 
\begin{equation}\label{eq:Xsum}
\mathcal X = \Delta_{0}(\mathcal T_h)\oplus\Oplus_{\ell = 1}^n \Oplus_{f\in \Delta_{\ell}(\mathcal T_h)} \mathcal X_{\mathring{f}},
\end{equation}
where $\Delta_{\ell}(\mathcal T_h)$ denotes the set of $\ell$-dimensional subsimplices of $\mathcal T_h$.
In implementation, computation of local matrices on each simplex is based on \eqref{eq:Xunion} while to assemble a matrix equation, \eqref{eq:Xsum} should be used. As Lagrange element is globally continuous, the indexing of interpolation points on vertices, edges, faces should be unique and a mapping from the local index to the global index is needed. The index mapping from \eqref{eq:Xunion} to \eqref{eq:Xsum} will be discussed in Section \ref{sec:implementation}.

\section{Geometric Decompositions of Lagrange Elements}\label{sec:lagrange}
In this section, we present a geometric decomposition for Lagrange finite elements on $n$-dimensional simplices. We introduce the concept of Lagrange interpolation basis functions, where function values at interpolation points serve as degrees of freedom. 

\subsection{Geometric decomposition}
For the polynomial space $\mathbb P_k(T)$ with $k\geq 1$ on an $n$-dimensional simplex $T$, we have the following geometric decomposition of Lagrange element~\cite[(2.6)]{ArnoldFalkWinther2009} and a proof can be found in~\cite{Chen;Huang:2021Geometric}. The integral at a vertex is understood as the function value at that vertex and $\mathbb P_k({\bs x}) = \mathbb R$. When $k<0$, $\mathbb P_k(T) = \varnothing$.

\begin{theorem}[Geometric decomposition of Lagrange element]\label{thm:Lagrangedec}
For the polynomial space $\mathbb P_k(T)$ with $k\geq 1$ on an $n$-dimensional simplex $T$, we have the following decomposition 
\begin{align}
\label{eq:Prdec}
\mathbb P_k(T) &= \Oplus_{\ell = 0}^n\Oplus_{f\in \Delta_{\ell}(T)} b_f\mathbb P_{k - (\ell +1)} (f).
\end{align}
The function $u\in \mathbb P_k(T)$ is uniquely determined by DoFs
\begin{equation}\label{eq:dofPr}
\int_f u \, p \dd s, \quad \quad~p\in \mathbb P_{k - (\ell +1)} (f), f\in \Delta_{\ell}(T), \ell = 0,1,\ldots, n.
\end{equation}
\end{theorem}


Introduce the bubble polynomial space of degree $k$ on a sub-simplex $f$ as
$$
\mathbb B_{k}( f) := b_f\mathbb P_{k - (\ell +1)} (f), \quad f\in \Delta_{\ell}(T), 1\leq \ell \leq n.
$$
When $\ell \geq k$, $\mathbb B_k(f) = \varnothing$.
It is called a bubble space as
$$
\tr^{\grad} u := u|_{\partial f} = 0, \quad u\in \mathbb B_{k}( f).
$$
Then we can write \eqref{eq:Prdec} as
\begin{equation}\label{eq:Lagdecbubble}
\mathbb P_k(T) = \mathbb P_1(T)\oplus \Oplus_{\ell = 1}^n\Oplus_{f\in \Delta_{\ell}(T)} \mathbb B_{k}( f).
\end{equation}
That is a polynomial of degree $k$ can be decomposed into a linear polynomial plus bubbles on edges, faces, and all sub-simplexes. 

Based on a conforming triangulation $\mathcal T_h$, the $k$-th order Lagrange finite element space $V_k^{\rm L}(\mathcal T_h)$ is defined as 
$$
V_k^{\rm L}(\mathcal T_h) = \{ v\in C(\Omega): v|_{T}\in \mathbb P_k(T),  T\in \mathcal T_h, \text{ and DoFs \eqref{eq:dofPr} are single valued}\},
$$
and will have a geometric decomposition
\begin{equation}\label{eq:VLdec}
V_k^{\rm L}(\mathcal T_h) = V_1^{\rm L}(\mathcal T_h) \,\oplus\, \Oplus_{\ell = 1}^n \Oplus_{f\in \Delta_{\ell}(\mathcal T_h)} \mathbb B_k(f).
\end{equation}
Here we extend the polynomial on $f$ to each element $T$ containing $f$ by the Bernstein form and extension of multi-index; see $E(\alpha)$ defined in \eqref{eq:Ealpha}. Consequently the dimension of $V_k^{\rm L}(\mathcal T_h)$ is
$$
\dim V_k^{\rm L}(\mathcal T_h) = \sum_{\ell = 0}^n |\Delta_{\ell}(\mathcal T_h)| { k - 1 \choose \ell},
$$
where $|\Delta_{\ell}(\mathcal T_h)|$ is the cardinality of  number of $\Delta_{\ell}(\mathcal T_h)$, i.e., the number of $\ell$-dimensional simplices in $\mathcal T_h$. We understand ${ k - 1 \choose \ell} = 0$ if $\ell > k-1$. That is the degree of the polynomial dictates the dimension of the sub-simplex in the geometric decompositions \eqref{eq:Prdec} and \eqref{eq:VLdec}. 


The geometric decomposition \eqref{eq:Prdec} can be naturally extended to vector Lagrange
elements. For $k\geq 1$, define 
\begin{equation*}
\mathbb B_{k}^n(f) := b_f\mathbb P_{k - (\ell +1)} (f)\otimes \mathbb R^n.
\end{equation*}
Clearly we have
\begin{align}\label{eq:Pkvecdec}
\mathbb P_k^n(T)&= \mathbb P_1^n(T)\oplus \Oplus_{\ell = 1}^n\Oplus_{f\in \Delta_{\ell}(T)} \mathbb B_{k}^n(f).
\end{align}
For an $f\in \Delta_{\ell} (T)$, we choose a $t-n$ coordinate 
$\{\bs t_i^f, \bs n_j^f, i=1, \dots, \ell, j = 1, \dots, n-\ell\}$ so that
\begin{itemize}
\item $\mathscr T^f := \textrm{span}\{ \bs t_1^f, \ldots, \bs t_{\ell}^f \}, $ is the tangential plane of $f$;

\item $\mathscr{N}^f :=  \textrm{span}\{ \bs n_1^f, \ldots, \bs n_{n - \ell}^f \} $ is the normal plane of $f$. 
\end{itemize}
%
When $\ell = 0$, i.e., for vertices, no tangential component, and for $\ell = n$, no normal component. 
We abbreviate $\bs n_1^F$ as $\bs n_F$ for $F\in\Delta_{n-1}(T)$, and $\bs t_1^e$ as $\bs t_e$ for $e\in\Delta_{1}(T)$.
We have the trivial decompositions
\begin{equation}\label{eq:tndec}
\mathbb R^n = \mathscr T^{f}\oplus \mathscr N^f, \quad \mathbb B_{k}^n(f) = \left [\mathbb B_{k}(f)\otimes \mathscr T^{f}\right ]\Oplus \left [\mathbb B_{k}(f)\otimes\mathscr N^f \right ].
\end{equation}
Restricted to an $\ell$-dimensional sub-simplex $f\in \Delta_{\ell}(T)$,  define
$$ 
\mathbb B_{k}^{\ell}(f) := \mathbb B_k(f) \otimes \mathscr{T}^f,
$$ 
which is a space of $\ell$-dimensional vectors on the tangential space $\mathscr{T}^f$ with vanishing trace $\tr^{\grad}$ on $\partial f$. 


When move to a triangulation $\mathcal T_h$, we shall call a basis of $\mathscr T^f$ or $\mathscr N^f$ is global if it depends only on $f$ not the element $T$ containing $f$. Otherwise it is called local and may vary in different elements.



\subsection{Lagrange interpolation basis functions}\label{sec:lagrangebasis}
Previously DoFs \eqref{eq:dofPr} are given by moments on sub-simplexes. Now we present a set of DoFs as function values on the interpolation points and give its dual basis for the $k$-th order Lagrange element on an $n$-simplex.

\begin{lemma}[Lagrange interpolation basis functions~\cite{nicolaides1972class}]\label{le:li}
A basis function of the $k$-th order Lagrange finite element space on $T$ is:
$$
\phi_{\boldsymbol \alpha}(\boldsymbol x) = 
\frac{1}{\boldsymbol \alpha!} 
\prod_{i=0}^{n}\prod_{j =0}^{\alpha_i - 1} (k\lambda_i(\boldsymbol x) - j), \quad 
\boldsymbol \alpha \in \mathbb T^n_k,
$$
with the DoFs defined as the function value at the interpolation points:
$$
N_{\bs \alpha} (u) = u(\boldsymbol x_{\alpha}), \quad \boldsymbol x_{\alpha} \in \mathcal X_{T}.
$$
\end{lemma}
\begin{proof}
It is straightforward to verify the duality of the basis and DoFs 
$$
N_{\bs \beta}(\phi_{\bs \alpha}) = \phi_{\bs \alpha}(\bs x_{\beta}) = \delta_{\bs \alpha,\bs \beta} = 
\begin{cases}
1& \text{ if } \bs \alpha = \bs \beta\\
0 & \text{ otherwise }
\end{cases}.
$$
As $$| \mathbb T^{n}_k | = {n + k \choose k} = \dim \mathbb P_k(T),$$ $\{\phi_{\boldsymbol \alpha}, \bs \alpha\in  \mathbb T^{n}_k\}$ is a basis of $\mathbb P_k(T)$ and $\{N_{\boldsymbol \alpha}, \bs \alpha\in  \mathbb T^{n}_k\}$ is a basis of the dual space $\mathbb P_k^*(T)$.
\end{proof}

Given a triangulation $\mathcal T_h$ and
degree $k$, recall that
$$
\mathcal X_{\mathcal T_h} = \bigcup_{T\in\mathcal T_h} \mathcal X_T = \Delta_{0}(\mathcal T_h)\oplus\Oplus_{\ell = 1}^n \Oplus_{f\in \Delta_{\ell}(\mathcal T_h)} \mathcal X_{\mathring{f}}.
$$
Denote by
$$
\mathbb T_k^n(\mathcal T_h) := \Oplus_{\bs x_i\in \Delta_0(\mathcal T_h)}\mathbb T_k^{0}(\bs x_i) \oplus \Oplus_{\ell = 1}^n \Oplus_{f\in \Delta_{\ell}(\mathcal T_h)} \mathbb T_k^{\ell}(\mathring{f}).
$$
For a lattice $\alpha \in \mathbb T_k^{\ell}(\mathring{f})$, we use extension operator $E$ defined in \eqref{eq:Ealpha} to extend $\alpha$ to each simplex $T$ containing $f$. We also extend the polynomial on $f$ to $T$ by the Bernstein form.
Hereafter, for simplicity, we set $\mathcal X_{\mathring{\texttt{v}}}:=\mathcal X_{\texttt{v}}$ and $\mathbb T_k^{0}(\mathring{\texttt{v}}):=\mathbb T_k^{0}(\texttt{v})$ for vertex $\texttt{v}\in\Delta_{0}(\mathcal T_h)$.

\begin{theorem}[DoFs of Lagrange finite element on $\mathcal T_h$]
A basis for the $k$-th Lagrange finite element space $V_k^{\rm L}(\mathcal T_h)$ 
is 
$$\{\phi_{\boldsymbol \alpha}, \bs \alpha\in  \mathbb T^{n}_k(\mathcal T_h)\}$$ with DoFs
$$
N_{\bs \alpha} (u) = u(\boldsymbol x_{\alpha}), \quad \boldsymbol x_{\alpha} \in \mathcal X_{\mathcal T_h}.
$$
\end{theorem}
\begin{proof}
For $F\in \Delta_{n-1}(\mathcal T_h)$, thanks to Lemma~\ref{le:li}, $\phi_{\boldsymbol \alpha}|_F$ is uniquely determined by DoFs $\{u(\boldsymbol x_{\alpha}), \boldsymbol x_{\alpha} \in \mathcal X_{F}\}$ on face $F$, hence $\phi_{\boldsymbol \alpha}\in V_k^{\rm L}(\mathcal T_h)$. Clearly the cardinality of $\{\phi_{\boldsymbol \alpha}, \bs \alpha\in  \mathbb T^{n}_k(\mathcal T_h)\}$ is same as the dimension of space $V_k^{\rm L}(\mathcal T_h)$. Then we only need to show these functions are linearly independent,  which follows from the fact $N_{\bs \beta}(\phi_{\bs \alpha}) = \phi_{\bs \alpha}(\bs x_{\beta}) = \delta_{\bs \alpha,\bs \beta}$ for $\bs \alpha, \bs \beta\in\mathbb T^{n}_k(\mathcal T_h)$.
\end{proof}

We now generalize the basis for a scalar Lagrange element to a vector Lagrange element.
For an interpolation point $\bs x\in \mathcal X_{T}$, let $\{\bs e_i^{\bs x}, i=0,\ldots, n-1\}$ be a basis of $\mathbb R^n$, and its dual basis is denoted by $\{\hat{\bs e}_i^{\bs x}, i=0,\ldots, n-1\}$, i.e.,
$$
(\hat{\bs e}_i^{\bs x}, \bs e_j^{\bs x}) = \delta_{i,j},
$$
where $(\boldsymbol a,\boldsymbol b) = \boldsymbol a\cdot \boldsymbol b$ is the inner product of two vectors in $\mathbb R^n$.
When $\{\bs e_i^{\bs x}, i=0,\ldots, n-1\}$ is orthonormal, its dual basis is itself. 

\begin{corollary}\label{cor:basisfunctions}
    A polynomial function $\bs u \in \mathbb P_k^n(T)$ can be uniquely determined by the DoFs:
  $$
N^i_{\bs \alpha }(\bs u):=  \bs u(\bs x_{\bs \alpha})\cdot \bs e^{\bs x_{\bs \alpha}}_i, \quad \bs x_{\bs \alpha} \in \mathcal X_{T}, i = 0,\ldots, n-1.
  $$
  The basis function on $T$ dual to this set of DoFs can be 
  explicitly written as:
  $$
  \bs{\phi}_{\bs \alpha}^i(\bs x) = 
  \phi_{\bs\alpha}(\bs x) \hat{\bs e}^{\bs x_{\bs \alpha}}_i, \quad 
  \boldsymbol \alpha \in \mathbb T^n_k, i = 0, \ldots, n-1.
  $$
\end{corollary}
 \begin{proof}
It is straightforward to verify the duality
$$
N^j_{\bs \beta }(\bs{\phi}_{\bs \alpha}^i) =   \bs{\phi}_{\bs \alpha}^i(\bs x_{\beta}) \cdot \bs e^{\bs x_{\beta}}_j  =   \phi_{\bs \alpha}(\bs x_{\beta}) \hat{\bs e}^{\bs x_{\bs \alpha}}_i \cdot \bs e^{\bs x_{\beta}}_j  = \delta_{i,j} \delta_{\bs \alpha,\bs \beta},
$$ 
for $\bs \alpha, \bs \beta\in  \mathbb T^n_k, i,j = 0,\ldots, n-1.$
\end{proof}
If the basis $\{\bs e_i^{f}, i=0,\ldots, n-1\}$ is global in the sense that it is independent of element $T$ containing $\bs x$, we get the continuous vector Lagrange elements. Choose different basis will give different continuity.

\section{Geometric Decompositions of Face Elements}\label{sec:facefem}
Define $H(\div, \Omega):=\{\bs v\in L^2(\Omega;\mathbb R^n): \div \bs v\in L^2(\Omega)\}.$ For a subdomain $K\subseteq \Omega$, the trace operator for the div operator is
$$
{\rm tr}_K^{\div} \boldsymbol v = \boldsymbol n\cdot \boldsymbol v|_{\partial K} \quad \textrm{ for }\;\;\boldsymbol{v}\in H(\div, \Omega),
$$
where $\bs n$ denotes the outwards unit normal vector of $\partial K$. 
Given a triangulation $\mathcal T_h$ and a piecewise smooth function $\bs u$, it is well known that $\bs u\in H(\div, \Omega)$ if and only if $\bs n_F \cdot \bs u$ is continuous across all faces $F\in \Delta_{n-1}(\mathcal T_h)$, which can be ensured by having DoFs on faces. 
An $H(\div)$-conforming finite element is thus also called a face element.


\subsection{Geometric decomposition}
Define the polynomial $\div$ bubble space
$$
\mathbb B_k(\div, T) = \ker(\tr_T^{\div})\cap \mathbb P_k^n(T).
$$
Recall that $\mathbb B_{k}^{\ell}(f) = \mathbb B_k(f) \otimes \mathscr{T}^f$ consists of bubble polynomials on the tangential plane of $f$. For $\bs u\in \mathbb B_{k}^{\ell}(f)$, as $\bs u$ is on the tangent plane, $\bs u\cdot \bs n_F = 0$ for $f\subseteq F$. When $f\not\subseteq F$, $b_f|_F =0$. So $\mathbb B_{k}^{\ell}(f)\subseteq \mathbb B_k(\div, T)$
for $k\geq 2$ and $\dim f \geq 1$. In~\cite{Chen;Huang:2021Geometric}, 
we have proved that the $\div$-bubble polynomial space has the following decomposition.
\begin{lemma}\label{lem:divbubble}
For $k\geq 2$,
$$
\mathbb B_k(\div, T) = \Oplus_{\ell=1}^n\Oplus_{f\in\Delta_{\ell}(T)}\mathbb B_{k}^{\ell}(f).
$$
\end{lemma}
Notice that as no tangential plane on vertices, there is no $\div$-bubble associated to vertices and consequently the degree of a $\div$-bubble polynomial is greater than or equal to $2$. 
%
Next we present two geometric decompositions of a $\div$-element. 


\begin{theorem}\label{thm:BDMgeodecomp}
For $k\geq 1$, we have
\begin{align}
\label{eq:divpolynomialdecomp}
\mathbb P_k^n(T) &= \mathbb P_1^n(T) \oplus \big(\Oplus_{\ell=1}^{n-1}\Oplus_{f\in\Delta_{\ell}(T)}(\mathbb B_{k}( f) \otimes \mathscr{N}^f)\big)\oplus \mathbb B_k(\div, T),\\
\label{eq:divpolynomialdecomp2}
\mathbb P_k^n(T) &= \big(\Oplus_{F\in\Delta_{n-1}(T)} \left ( \mathbb P_k(F)\bs n_F \right )\big)\oplus \mathbb B_k(\div, T).
\end{align}
\end{theorem}
\begin{proof}
The first decomposition \eqref{eq:divpolynomialdecomp} is a rearrangement of \eqref{eq:Pkvecdec} by merging the tangential component $\mathbb B_{k}^{\ell}( f) $ into the bubble space $\mathbb B_k(\div, T)$. 

Next we prove the decomposition \eqref{eq:divpolynomialdecomp2}.
For an $\ell$-dimensional, $0\leq \ell\leq n-1$, sub-simplex $f\in\Delta_{\ell}(T)$, we choose the $n-\ell$ face normal vectors $\{\bs n_{F}: F\in \Delta_{n-1}(T), f\subseteq F\}$ as the basis of $\mathscr{N}^f$. Therefore  we have
$$
\mathbb B_{k}( f) \otimes \mathscr{N}^f =\Oplus_{F\in \Delta_{n-1}(T), f\subseteq F}\mathbb B_{k}( f)\bs n_F.
$$
Then \eqref{eq:divpolynomialdecomp} becomes
\begin{equation*}
\mathbb P_k^n(T) = \mathbb P_1^n(T) \oplus \big(\Oplus_{\ell=1}^{n-1}\Oplus_{f\in\Delta_{\ell}(T)}\Oplus_{F\in \Delta_{n-1}(T), f\subseteq F}\mathbb B_{k}( f)\bs n_F\big)\oplus \mathbb B_k(\div, T).
\end{equation*}
At a vertex $\texttt{v}$, we choose $\{\bs n_{F}: F\in \Delta_{n-1}(T), \texttt{v}\in F\}$ as the basis of $\mathbb R^n$ and write 
\begin{align*}  
\mathbb P_1^n(T) 
&= \Oplus_{{\bs x}\in \Delta_{0}(T)}\Oplus_{F\in\Delta_{n-1}(T), {\bs x}\in\Delta_0(F)}{\rm span}\{ \lambda_{{\bs x}}\bs n_F\} \\
&= \Oplus_{F\in\Delta_{n-1}(T)}\Oplus_{{\bs x}\in \Delta_{0}(F)}{\rm span}\{ \lambda_{{\bs x}}\bs n_F\}\\
&= \Oplus_{F\in\Delta_{n-1}(T)} \mathbb P_1(F)\bs n_F.
\end{align*}
Then by swapping the ordering of $f$ and $F$ in the direct sum, i.e.,
$$
\Oplus_{\ell=1}^{n-1}\Oplus_{f\in\Delta_{\ell}(T)}\Oplus_{F\in \Delta_{n-1}(T), f\subseteq F} \to \Oplus_{F\in \Delta_{n-1}(T)}\Oplus_{\ell=1}^{n-1}\Oplus_{f\in\Delta_{\ell}(F)},
$$
and using the decomposition \eqref{eq:Prdec} of Lagrange element, we obtain the decomposition \eqref{eq:divpolynomialdecomp2}.
%
\end{proof}

In decomposition \eqref{eq:divpolynomialdecomp}, we single out $\mathbb P_1^n(T)$ to emphasize an $H(\div)$-conforming element can be obtained by adding $\div$-bubble and normal component on sub-simplexes starting from edges. In~\eqref{eq:divpolynomialdecomp2}, we group all normal components facewisely which leads to the classical BDM element.

As $H(\div,\Omega)$-conforming elements require the normal continuity across each $F$, the normal vector $\bs n_F$ is chosen globally. Namely $\bs n_F$ depends on $F$ only. It may coincide with the outwards or inwards normal vector for an element $T$ containing $F$. On the contrary, all tangential basis for $\mathscr T^f$ are local and thus the tangential component are multi-valued and merged to the element-wise div bubble function $\mathbb B_k(\div, T)$.

\subsection{A nodal basis for the BDM element}
For the efficient implementation, we employ the DoFs of the function values at interpolation nodes and combine with $t-n$ decompositions to present a nodal basis for the BDM element.  
 
Given an $f\in \Delta_{\ell}(T)$, we choose $\{\bs n_F, F\in \Delta_{n-1}( T), f\subseteq F\}$ as the basis for its normal plane $\mathcal N^f$ and an arbitrary basis $\{ \bs t_i^f, i=1,\ldots, \ell \}$ for the tangential plane. We shall choose  $\{\hat{\bs n}_{F_i}\in \mathcal N^f, i\in f^*\}$ a basis of $\mathcal N^f$  dual to $\{\bs n_{F_i}\in \mathcal N^f, i\in f^*\}$, i.e.,
$$
(\bs n_F, \hat{\bs n}_{F'}) = \delta_{F, F'}, \quad F, F'\in \Delta_{n-1}(T), f\in F\cap F'.
$$
Similarly  choose a basis $\{\hat{\bs t}_i^f\}$ dual to $\{\bs t_j^f\}$.  

We can always choose an orthonormal basis for the tangential plane $\mathscr T^f$ but for the normal plane $\mathscr N^f$ with basis $\{\bs n_{F_i}\in \mathcal N^f, i\in f^*\}$, we use Lemma \ref{lm:dualnormal} to find its dual basis. 

For $f\in \Delta_{\ell}(T)$ and $e\in\partial f$, let $\bs n_{f}^e$ be an unit normal vector of $e$ but tangential to $f$. When $\ell=1$,$f$ is an edge and $e$ is a vertex. Then $\bs n_{f}^e$ is the edge vector of $f$. 
Using these notations we can give an explicit expression of the dual basis $\{\hat{\bs n}_{F_i}\in \mathcal N^f, i\in f^*\}$.

\begin{lemma}\label{lm:dualnormal}
For $f\in \Delta_{\ell}(T)$, 
\begin{equation}\label{eq:dualnF}
\{\hat{\bs n}_{F_i}=\frac{1}{\bs n_{f+i}^f\cdot\bs n_{F_i}}\bs n_{f+i}^f \quad i\in f^*\},
\end{equation}
where $f+i$ denotes the $(\ell+1)$-dimensional face in $\Delta_{\ell+1}(T)$ with vertices $f(0),\dots, f(\ell)$ and $i$ for $i\in f^*$, is a basis of $\mathscr N^f$ dual to $\{\bs n_{F_i}\in \mathcal N^f, i\in f^* \}$.
\end{lemma}
\begin{proof}
Clearly $\bs n_{f+i}^f\in \mathscr N^f$ for $i\in f^*$. It suffices to prove
\begin{equation*}
 \bs n_{f+i}^f\cdot\bs n_{F_j}=0 \quad \textrm{ for } i,j\in f^*, i\neq j,
\end{equation*}
which follows from $\bs n_{f+i}^f\in\mathscr T^{f+i}$ and $f+i\subseteq F_j$.
\end{proof}

By Corollary \ref{cor:basisfunctions}, we obtain a nodal basis for BDM face elements.

\begin{theorem}
For each $f\in \Delta_{\ell}(T)$, we choose $\{\bs e_i^{f}, i=0,\ldots, n-1\} = \{ \bs t_i^f, i=1,\ldots, \ell, \; \; \bs n_{F_i}, i\in f^*\}$ and its dual basis $\{\hat{\bs e}_i^{f}, i=0,\ldots, n-1\} = \{ \hat{\bs t}_i^f, i=1,\ldots, \ell, \; \; \hat{\bs n}_{F_i}, i\in f^*\}$.
A basis function of the $k$-th order BDM element space on $T$ is:
$$
\{\phi_{\bs\alpha}(\bs x)\hat{\bs e}^f_{i},  i=0,1,\ldots, n-1, \; \;  \alpha \in \mathbb T_k^{\ell}(\mathring{f})\}_{f \in \Delta_{\ell}(T), \ell =0,1,\ldots, n},
$$
with the DoFs at the interpolation points defined as:
\begin{align}\label{BDMnodalDoFs}
\{\boldsymbol{u}(\boldsymbol{x}_{\alpha}) \cdot \bs e^f_i,  i=0,1,\ldots, n-1, \; \;  \boldsymbol x_{\alpha} \in \mathcal X_{\mathring{f}} \}_{f \in \Delta_{\ell}(T), \ell =0,1,\ldots, n}.
\end{align}
\end{theorem}



By choosing a global normal basis in the sense that $\bs n_F$ depending only on $F$ not the element containing $F$, we impose the continuity on the normal direction. We choose a local $\bs t^f$, i.e., for different element $T$ containing $f$, $\phi_{\bs\alpha}(\bs x)\bs t^f(T)$
is different, then no continuity is imposed for the tangential direction. 

Define the global finite element space
\begin{align*}
V_{h,k}^{\div}:=&\{\bs u\in L^2(\Omega;\mathbb R^n): \bs u|_T\in\mathbb P_k(T; \mathbb R^n) \textrm{ for each } T\in\mathcal T_h, \\
 &\textrm{ all the DoFs } \boldsymbol{u}(\boldsymbol{x}_{\alpha}) \cdot \bs n_F \text{ in \eqref{BDMnodalDoFs} across $F\in\Delta_{n-1}(\mathcal T_h)$ are single-valued}\}.
\end{align*}
By Theorem~\ref{thm:BDMgeodecomp}, $V_{h,k}^{\div}\subset H(\div,\Omega)$ and is equivalent to the BDM space.

The novelty is that we only need a basis of Lagrange element which is well documented; see Section \ref{sec:lagrangebasis}. Coupling with different $t-n$ decomposition at different sub-simplex, we obtain the classical face elements. This concept has been explored in the works of \cite{DeSiqueiraDevlooGomes2010, DeSiqueiraDevlooGomes2013} and \cite{CastroDevlooFariasGomesEtAl2016}, focusing on the implementation of the $H(\div)$ element in two and three dimensions. Additionally, the adaptation of a 2D $H(\curl)$ element through rotation is discussed in \cite{DeSiqueiraDevlooGomes2010, DeSiqueiraDevlooGomes2013}. As we will demonstrate in the subsequent section, extending the edge element to higher dimensions presents significant challenges. This extension necessitates a comprehensive characterization of the $\curl$ operator and its associated polynomial bubble space.

\section{Geometric Decompositions of Edge Elements}\label{sec:edgefem}
In this section we present geometric decompositions of $H(\curl)$-conforming finite element space on an $n$-dimensional simplex. We first generalize the $\curl$ differential operator to $n$ dimensions and study its trace which motivates our decomposition. We then give an explicit basis dual to function values at interpolation points.

\subsection{Differential operator and its trace}
Denote by $\mathbb S$ and $\mathbb K$ the subspace of symmetric matrices and skew-symmetric matrices of $\mathbb R^{n\times n}$, respectively. For a smooth vector function $\boldsymbol{v}$, define
$$
\curl\boldsymbol{v}:=2\skw(\grad\boldsymbol{v}) = \grad\boldsymbol{v} - (\grad\boldsymbol{v})^{\intercal},
$$
which is a skew-symmetric matrix function. In two dimensions, for $\boldsymbol{v}=(v_1, v_2)^{\intercal}$,
$$
\curl\boldsymbol{v}=\begin{pmatrix}
0 & \partial_{x_2}v_1-\partial_{x_1}v_2 \\
\partial_{x_1}v_2-\partial_{x_2}v_1 & 0
\end{pmatrix}=\mskw(\rot\boldsymbol{v}),
$$
where
$
\mskw u:=\begin{pmatrix}
0 & -u \\
u & 0
\end{pmatrix}$ and $\rot\boldsymbol{v}:=\partial_{x_1}v_2-\partial_{x_2}v_1$.
In three dimensions, for $\boldsymbol{v}=(v_1, v_2, v_3)^{\intercal}$,
$$
\curl\boldsymbol{v}=\begin{pmatrix}
0 & \partial_{x_2}v_1-\partial_{x_1}v_2 & \partial_{x_3}v_1-\partial_{x_1}v_3 \\
\partial_{x_1}v_2-\partial_{x_2}v_1 & 0 & \partial_{x_3}v_2-\partial_{x_2}v_3 \\
\partial_{x_1}v_3-\partial_{x_3}v_1 & \partial_{x_2}v_3-\partial_{x_3}v_2 & 0 
\end{pmatrix}=\mskw(\nabla\times\boldsymbol{v}),
$$
where
$
\mskw \boldsymbol{u}:=\begin{pmatrix}
 0 & -u_3 & u_2 \\
u_3 & 0 & - u_1\\
-u_2 & u_1 & 0
\end{pmatrix}$ with $\boldsymbol{u}=(u_1, u_2, u_3)^{\intercal}$.
Hence we can identify $\curl\boldsymbol{v}$ as scalar $\rot\boldsymbol{v}$ in two dimensions, and vector $\nabla\times\boldsymbol{v}$ in three dimensions. In general, $\curl \bs u$ is understood as a skew-symmetric matrix.

Define Sobolev space
\begin{equation*}
H(\curl,\Omega):=\{\boldsymbol{v}\in L^2(\Omega;\mathbb R^n): \curl\boldsymbol{v}\in L^2(\Omega;\mathbb K)\}.
\end{equation*}
Given a face $F\in \Delta_{n-1}(T)$, define the trace operator of $\curl$ as
$$
{\rm tr}_F^{\curl}\boldsymbol v=2\skw(\boldsymbol{v}\boldsymbol{n}_F^{\intercal})|_F = (\boldsymbol{v}\boldsymbol{n}_F^{\intercal} - \boldsymbol{n}_F\boldsymbol{v}^{\intercal})|_F.
$$
Recall that for two column vectors $\boldsymbol a$ and $\boldsymbol b$, $\boldsymbol a \boldsymbol b^{T}$ is a matrix.
We define $\tr^{\curl}$ as a piecewise defined operator as $$(\tr^{\curl}\bs v) |_F = \tr^{\curl}_F\bs v, \quad F\in \Delta_{n-1}(T).$$

For a vector $\bs v\in \mathbb R^n$ and an $(n-1)$-dimensional face $F$, the tangential part of $\boldsymbol v$ on $F$ is
$$
\Pi_F\boldsymbol v := \boldsymbol v|_F - (\boldsymbol v|_F\cdot\boldsymbol n_F)\boldsymbol n_F = \sum_{i=1}^{n-1}(\boldsymbol v|_F\cdot\boldsymbol t_{i}^F)\boldsymbol t_{i}^F,
$$ 
where $\{ \bs t_{i}^F, i=1,\ldots, n-1 \}$ is an orthonormal basis of $F$. 
As we treat $\curl \bs v$ as a matrix, so is the trace ${\rm tr}_F^{\curl}\boldsymbol v$, while the tangential component of $\bs v$ is a vector. Their relation is given in the following lemma.

\begin{lemma}
For face $F\in \Delta_{n-1}(T)$, 
we have 
\begin{equation}\label{eq:tangentialtracetangentialpart}
{\rm tr}_F^{\curl}\boldsymbol v = 2\skw\big((\Pi_F\boldsymbol v)\boldsymbol{n}_F^{\intercal}\big), \quad \Pi_F\boldsymbol v = ({\rm tr}_F^{\curl}\boldsymbol v)\boldsymbol{n}_F.
\end{equation}
\end{lemma}
\begin{proof}
By the decomposition $\boldsymbol v|_F =\Pi_F\boldsymbol v + (\boldsymbol v|_F\cdot\boldsymbol n_F)\boldsymbol n_F$,
$$
{\rm tr}_F^{\curl}\boldsymbol v=2\skw\big((\Pi_F\boldsymbol v)\boldsymbol{n}_F^{\intercal} + (\boldsymbol v|_F\cdot\boldsymbol n_F)\boldsymbol n_F\boldsymbol{n}_F^{\intercal}\big) = 2\skw\big((\Pi_F\boldsymbol v)\boldsymbol{n}_F^{\intercal}\big),
$$
which implies the first identity.
Then by $\boldsymbol{n}_F^{\intercal}\boldsymbol{n}_F=1$ and $(\Pi_F\boldsymbol v)^{\intercal}\boldsymbol{n}_F=0$,
$$
({\rm tr}_F^{\curl}\boldsymbol v)\boldsymbol{n}_F=\big((\Pi_F\boldsymbol v)\boldsymbol{n}_F^{\intercal} - \boldsymbol{n}_F(\Pi_F\boldsymbol v)^{\intercal}\big)\boldsymbol{n}_F=\Pi_F\boldsymbol v,
$$
i.e. the second identity holds.
\end{proof}
Thanks to \eqref{eq:tangentialtracetangentialpart}, the vanishing tangential part $\Pi_F\boldsymbol v$ and the vanishing tangential trace $({\rm tr}_F^{\curl}\boldsymbol v)$ are equivalent.

\begin{lemma}
Let $\bs v \in L^2(\Omega;\mathbb R^n)$ and $\bs v|_T\in H^{1}(T;\mathbb R^n)$ for each $T\in \mathcal T_h$. Then $\bs v \in H(\curl,\Omega)$ if and only if $\,\Pi_F\boldsymbol v\mid_{T_1} = \Pi_F\boldsymbol v\mid_{T_2}\,$ for all interior face $F\in \Delta_{n-1}(\mathcal T_h)$, where $T_1$ and $T_2$ are two elements sharing $F$.
\end{lemma}
\begin{proof}
It is an immediate result of  Lemma 5.1 in~\cite{ArnoldFalkWinther2006} and \eqref{eq:tangentialtracetangentialpart}.
\end{proof}

\subsection{Polynomial bubble space}
Define the polynomial bubble space of degree $k$ for the $\curl$ operator as
$$
\mathbb B_k(\curl, T) = \ker(\tr^{\curl})\cap \mathbb P_k^n(T).
$$
For Lagrange bubble space $\mathbb B_k^n(T)$, all components of the vector function vanish on $\partial T$ and consequently on all sub-simplex with dimension $\leq n-1$. For $\bs u \in \mathbb B_k(\curl, T)$, only the tangential component vanishes, which will imply $\bs u$ vanishes on sub-simplex with dimension $\leq n-2$. 

\begin{lemma}\label{lm:curlbubbleface}
For $\bs u\in \mathbb B_k(\curl, T)$, it holds $\bs u|_f=\bs0$ for all $f\in\Delta_{\ell}(T), 0\leq \ell\leq n-2$. Consequently $\mathbb B_k(\curl, T)\subset \mathbb B_k^n(T) \oplus\Oplus_{F\in \Delta_{n-1}(T)}  \mathbb B_k^{n}(F)$.
\end{lemma}
\begin{proof}
It suffices to consider a sub-simplex $f\in \Delta_{n-2}(T)$. Let $F_1, F_2\in\Delta_{n-1}(T)$ such that $f=F_1\cap F_2$. By $\tr_{F_i}^{\curl}\bs u=\bs0$ for $i=1,2$, we have $\Pi_{F_i}\bs u = 0$ and consequently
$$
(\bs u\cdot\bs t_i^f)|_f=0, \;\; (\bs u\cdot\bs n_{F_1}^f)|_f=(\bs u\cdot\bs n_{F_2}^f)|_f=0\quad \textrm{ for } i=1,\ldots, n-2,
$$
where $\bs n_{F_i}^f$ is a normal vector $f$ sitting on $F_i$. As $\spa \{\bs t_1^f, \ldots, \bs t_{n-2}^f, \bs n_{F_1}^f, \bs n_{F_2}^f \} = \mathbb R^n$, we acquire $\bs u|_f=\bs0$. By the property of face bubbles, we conclude $\bs u$ is a linear combination of element bubble and $(n-1)$-dimensional face bubbles. 
\end{proof}

Obviously $\mathbb B_k^n(T)\subset \mathbb B_k(\curl, T)$. As $\tr^{\curl}$ contains the tangential component only, the normal component $\mathbb B_k(F)\bs n_F$ is also a $\curl$ bubble. The following result says their sum is precisely all $\curl$ bubble polynomials.  
\begin{theorem}\label{thm:curlbubbletracespacedecomp}
For $k\geq 1$, it holds that
\begin{equation}\label{eq:curlbubbledecomp}
\mathbb B_k(\curl, T) =  \mathbb B_k^n(T) \oplus \Oplus_{F\in \Delta_{n-1}(T)} \mathbb B_k(F)\bs n_F,
\end{equation}
and
\begin{equation}\label{eq:trNrf}
{\rm tr}^{\curl} : \mathbb P_1^n(T)\oplus \Oplus_{\ell = 1}^{n-2}\Oplus_{f\in \Delta_{\ell}(T)} \mathbb B_k^n(f)\oplus \Oplus_{F\in \Delta_{n-1}(T)}  \mathbb B_k^{n-1}(F) \to  {\rm tr}^{\curl} \mathbb P_k^n(T)
\end{equation} 
is a bijection.
\end{theorem}
\begin{proof}
It is obvious that
$$
\mathbb B_k^n(T) \oplus \Oplus_{F\in \Delta_{n-1}(T)} \mathbb B_k(F)\bs n_F \subseteq \mathbb B_{k}(\curl,T).
$$
Then apply the trace operator to the decomposition \eqref{eq:Pkvecdec}
 to conclude that 
the map $\tr^{\curl}$ in~\eqref{eq:trNrf} is onto. 

Now we prove it is also injective. Take a function $\bs u\in  \mathbb P_1^n(T)\oplus \Oplus_{\ell = 1}^{n-2}\Oplus_{f\in \Delta_{\ell}(T)} \mathbb B_k^n(f)\oplus \Oplus_{F\in \Delta_{n-1}(T)}  \mathbb B_k^{n-1}(F)$ and $\tr^{\curl} \bs u =\bs 0$. By Lemma~\ref{lm:curlbubbleface}, we can assume $\bs u=\sum\limits_{F\in \Delta_{n-1}(T)}\bs u_k^F$ with $\bs u_k^F \in \mathbb B_k^{n-1}(F)$.
Take $F\in \Delta_{n-1}(T)$. We have $\bs u|_F=\bs u_k^F|_F\in \mathbb B_k^{n-1}(F)$. Hence $(\bs u_k^F\cdot\bs t)|_F=(\bs u\cdot\bs t)|_F=0$ for any $\bs t\in\mathscr T^F$, which results in $\bs u_k^F=\bs0$. Therefore $\bs u = \bs0$.  


Once we have proved the map $\tr$ in \eqref{eq:trNrf} is bijection, we conclude \eqref{eq:curlbubbledecomp} from the decomposition \eqref{eq:Pkvecdec}. 
\end{proof}


%

We will use $\curl_f$ to denote the $\curl$ operator restricted to a sub-simplex $f$ with $\dim f \geq 1$. For $f\in\Delta_{\ell}(T), \ell=2,\ldots, n-1$, by applying Theorem~\ref{thm:curlbubbletracespacedecomp} to $f$, we have
\begin{equation}\label{eq:curlfbubble}
\mathbb B_k(\curl_f, f) = \mathbb B_{k}^{\ell}(f) \oplus \Oplus_{e\in \partial f}\mathbb B_{k}( e) \boldsymbol{n}_{f}^e.
\end{equation}
Notice that the $\curl_f$-bubble function is defined for $\ell \geq 2$ not including edges. 
Indeed, for an edge $e$ and a vertex ${\bs x}$ of $e$, $\boldsymbol{n}_{e}^{\bs x}$ is $\boldsymbol{t}_{e}$ if $\bs x$ is the ending vertex of $e$ and $-\boldsymbol{t}_{e}$ otherwise. Then for $\ell = 1$
\begin{equation}\label{eq:curlebubble}
\mathbb B_{k}( e) \bs t_e \oplus \Oplus_{{\bs x}\in \partial e} {\rm span}\{\lambda_{\bs x} \boldsymbol{n}_{e}^{\bs x}\} = \mathbb P_{k}(e)\bs t_e.
\end{equation}
which is the full polynomial not vanishing on $\partial e$. 

\subsection{Geometric decompositions}
For $e\in\Delta_{\ell}(T)$, we choose the basis $\{\bs n_{f}^e: f = e+i, i\in e^*\}$ for $\mathscr{N}^e$ and a basis $\{\bs t_i^e,  i=1,\ldots,\ell\}$ for $\mathscr{T}^e$. So we have the following geometric decompositions of $\mathbb P_k^n(T) $. 
\begin{theorem}
For $k\geq 1$, we have
\begin{align}
\label{eq:curlpolynomialdecomp}
\mathbb P_k^n(T) &= \mathbb P_1^n(T) \oplus \Oplus_{\ell=1}^n \Oplus_{e\in\Delta_{\ell}(T)}(\mathbb B_k(e)\otimes \spa\{\bs t_i^e\}_{i=1}^{\ell}) \oplus (\mathbb B_k(e)\otimes \spa\{\bs n_{e+i}^e, i\in e^*\}),\\
\label{eq:curlpolynomialdecomp2}
\mathbb P_k^n(T) &= \Oplus_{e\in \Delta_1(T)}\mathbb P_k(e)\bs t_e \oplus \Oplus_{\ell=2}^n\Oplus_{f\in\Delta_{\ell}(T)}\mathbb B_k(\curl_f, f).
\end{align}
\end{theorem}
\begin{proof}
Decomposition \eqref{eq:curlpolynomialdecomp} is the component form of decomposition \eqref{eq:Pkvecdec}. We can write $\mathbb B_k(e)\otimes \spa\{\bs n_{e+i}^e, i\in e^*\} = \Oplus_{f\in \Delta_{\ell+1}(T),e\subseteq f}\mathbb B_k( e)\boldsymbol{n}_{f}^e.$
Then in the summation \eqref{eq:curlpolynomialdecomp}, we shift the normal component one level up and switch the sum of $e$ and $f$:
\begin{align*}
&\Oplus_{\ell = 1}^n\Oplus_{e\in \Delta_{\ell}(T)} \big [ \mathbb B_{k}^{\ell}(e)\Oplus_{f\in \Delta_{\ell+1}(T),e\subseteq f}\mathbb B_k( e)\boldsymbol{n}_{f}^e \big ]\\
=&\Oplus_{e\in\Delta_{1}(T)}\mathbb B_k(e)\bs t_e \oplus \Oplus_{\ell=2}^n\Oplus_{f\in\Delta_{\ell}(T)}\big [\mathbb B_{k}^{\ell}(f) \Oplus_{e\subseteq \partial f}\mathbb B_k( e)\boldsymbol{n}_{f}^e \big ].
\end{align*}
Then by the characterization \eqref{eq:curlfbubble} of $\mathbb B_k(\curl_f, f)$ and \eqref{eq:curlebubble}, we get the decomposition \eqref{eq:curlpolynomialdecomp2}.
%
\end{proof}
Decomposition \eqref{eq:curlpolynomialdecomp2} is the counterpart of \eqref{eq:Lagdecbubble} for Lagrange element. 

\subsection{Tangential-Normal decomposition of the second family of edge elements}
Recall that for $e\in\Delta_{\ell-1}(T)$, the basis 
$\{\bs n_{f}^e: f\in \Delta_{\ell}(T), e\subseteq f\}$ of $\mathscr{N}^e$ is dual to the basis $\{\bs n_{F}: F\in \Delta_{n-1}(T), e\subseteq F\}$; see Lemma \ref{lm:dualnormal}.  

\begin{theorem}\label{thm:NedelecTNdecomp}
Take $\mathbb P_k^n(T)$ as the shape function space.
Then it is determined by the following DoFs
\begin{subequations}\label{eq:dofNedelec}
\begin{align}
\label{eq:dofvectoredgecurl}
\int_e \bs u\cdot \bs t\ p \dd s,& \quad p\in \mathbb P_{k} (e), e\in \Delta_1(T),\\
\label{eq:dofvectorfacecurl}
\int_f \bs u\cdot \bs p \dd s,& \quad \bs p\in \mathbb B_k(\curl_f, f),f\in\Delta_{\ell}(T), \ell=2,\ldots, n
\end{align}
\end{subequations}
\end{theorem}
\begin{proof}
%
Based on the decomposition \eqref{eq:curlpolynomialdecomp}, the shape function $\mathbb P_k^n(T)$ is determined by the following DoFs
\begin{subequations}
\label{eq:dofveccurlf}
\begin{align}
\label{eq:dofvectortedgecurlf}
\int_e \bs u\cdot \bs t\ p \dd s,& \quad p\in \mathbb P_{k} (e), e\in\Delta_1(T),\\
\label{eq:dofvectortcurlf}
\int_f (\bs u\cdot \bs t_i^f)\ p \dd s,& \quad p\in \mathbb P_{k - (\ell +1)} (f), i=1,\ldots, \ell, \\
\label{eq:dofvectorncurlf}
\int_e (\bs u\cdot \bs n_{f}^e)\ p \dd s, &\quad p\in \mathbb P_{k - \ell} (e), e\in\partial f 
\end{align}
\end{subequations}
for $f\in \Delta_{\ell}(T), \ell = 2,\ldots, n$.
By \eqref{eq:curlbubbledecomp} and \eqref{eq:curlfbubble},  DoFs \eqref{eq:dofveccurlf} are equivalent to DoFs \eqref{eq:dofNedelec}.
\end{proof}


\begin{remark}\rm
 The DoFs of the second kind N\'ed\'elec edge element in~\cite{Nedelec1986,ArnoldFalkWinther2006} are
\begin{align*}
\int_e \bs u\cdot \bs t\ p \dd s,& \quad p\in \mathbb P_{k} (e), e\in \Delta_1(T), \\
\int_f \bs u\cdot \bs p \dd s,& \quad \bs p\in \mathbb P_{k-\ell}^{\ell}(f)+(\Pi_f\bs x)\mathbb P_{k-\ell}(f), f\in\Delta_{\ell}(T), \ell=2,\ldots, n. 
\end{align*} 
There is an isomorphism between $\mathbb P_{k-\ell}^{\ell}(f)+(\Pi_f\bs x)\mathbb P_{k-\ell}(f)$ and $\mathbb B_k(\curl_f, f)$ for $f\in\Delta_{\ell}(T)$ with $\ell=2, \ldots, n$, that is $\mathbb B_k(\curl_f, f)$ is uniquely determined by DoF 
\begin{equation*}
\int_f \bs u\cdot \bs p \dd s, \quad \bs p\in \mathbb P_{k-\ell}^{\ell}(f)+(\Pi_f\bs x)\mathbb P_{k-\ell}(f),
\end{equation*}
whose proof can be found in Lemma 4.7 in~\cite{ArnoldFalkWinther2006}.
$\qed$
\end{remark}

Given an $e\in \Delta_{\ell-1}(T_h)$, we choose a global $\{\bs n_{f}^e, e\subseteq f\in\Delta_{\ell}(T_h)\}$ as the basis for the normal plane $\mathscr N^e$ and a global basis $\{\bs t_i^e\}$ of $\mathscr T^e$. 
 Define the global finite element space
\begin{align*}
V_{h,k}^{\rm curl}:=\{\bs u\in L^2(\Omega;\mathbb R^n): &\bs u|_T\in\mathbb P_k(T; \mathbb R^n) \textrm{ for each } T\in\mathcal T_h, \\
 &\textrm{ all the DoFs \eqref{eq:dofNedelec} are single-valued}\}.
\end{align*}

\begin{lemma}
We have $V_h^{\curl}\subset H(\curl,\Omega)$.
\end{lemma}
\begin{proof}
For an $F\in\Delta_{n-1}(T)$, DoFs \eqref{eq:dofveccurlf} related to $\Pi_F\bs u$ are
\begin{align*}
\int_e (\Pi_F\bs u\cdot \bs t_i^e)\ p \dd s,& \quad i=1,\ldots, \ell-1, p\in \mathbb P_{k - \ell} (e), e\in \Delta_{\ell-1}(F), \ell = 2,\ldots, n,\\
\int_e (\Pi_F\bs u\cdot \bs n_{f}^e)\ p \dd s, &\quad f\in\Delta_{\ell}(F), e\subseteq f, p\in \mathbb P_{k - \ell} (e), e\in \Delta_{\ell-1}(F), \ell = 1,\ldots, n-1,
\end{align*}
thanks to DoFs \eqref{eq:dofveccurlf}, which uniquely determine $\Pi_F\bs u$.
\end{proof}

\subsection{The second family of edge elements}\label{sec:nedelecbasis}


By Corollary \ref{cor:basisfunctions}, we obtain a nodal basis for the second-kind N\'ed\'elec edge element.

\begin{theorem}
For each $e\in \Delta_{\ell}(T), \ell = 0,1,\ldots, n$, we choose $\{\bs e_i^{e}, i=0,\ldots, n-1\} = \{ \bs t_i^e, i=1,\ldots, \ell, \; \; \; \bs n_{f}^e,  f= e+i, i\in e^*\}$ and its dual basis $\{\hat{\bs e}_i^{e}, i=0,\ldots, n-1\} = \{ \hat{\bs t}_i^e, i=1,\ldots, \ell, \bs n_{F_i}/ (\bs n_{F_i}\cdot \bs n^e_{e+i}), i\in e^*\}$.
A basis function of the $k$-th order second kind N\'ed\'elec edge element space on $T$ is:
$$
\{\phi_{\bs\alpha}(\bs x)\hat{\bs e}^e_{i},  i=0,1,\ldots, n-1, \; \;  \alpha \in \mathbb T_k^{\ell}(\mathring{e})\}_{e \in \Delta_{\ell}(T), \ell =0,1,\ldots, n},
$$
with the DoFs at the interpolation points defined as:
\begin{align}\label{edgenodalDoFs}
\{\boldsymbol{u}(\boldsymbol{x}_{\alpha}) \cdot \bs e^e_i,  i=0,1,\ldots, n-1, \; \;  \boldsymbol x_{\alpha} \in \mathcal X_{\mathring{e}} \}_{e \in \Delta_{\ell}(T), \ell =0,1,\ldots, n}.
\end{align}
\end{theorem}

Notice that the basis $\{\bs t^e_i\}$ depends only on $e$ and $\bs n_{f}^e$ depends on $e$ and $f$ but independent $T$. By asking the corresponding DoFs single valued, we obtain the tangential continuity. 
We can define the edge finite element space as
\begin{align*}
V_{h,k}^{\rm curl}:=\{&\bs u\in L^2(\Omega;\mathbb R^n): \bs u|_T\in\mathbb P_k(T; \mathbb R^n) \textrm{ for each } T\in\mathcal T_h, \textrm{ the DoFs } \boldsymbol{u}(\boldsymbol{x}_{\alpha}) \cdot \bs t^e_i, \boldsymbol{u}(\boldsymbol{x}_{\alpha}) \cdot \bs n^e_f \\
 &\text{ in \eqref{edgenodalDoFs} are single-valued across all $e$ and $f$ with $\dim e, \dim f \leq n-1$}\}.
\end{align*}

\subsubsection{2D basis on triangular meshes}
Let $T$ be a triangle, for lattice point $\boldsymbol{x}$ located in different sub-simplices, 
we shall choose different frame $\{\bs e_{\boldsymbol x}^0, \bs e_{\bs x}^1\}$ 
at $\bs x$ and its dual frame 
$\{\hat {\bs e}_{\bs x}^0, \hat{\bs e}_{\bs x}^1\}$ 
as follows:
\begin{enumerate}
  \item If $\bs x \in \Delta_0(T)$, assume the two adjacent edges are $e_0$ 
    and $e_1$, then
  $$
  \bs e_{\bs x}^0 = \bs t_{e_0}, \quad 
  \bs e_{\bs x}^1 = \bs t_{e_1}, \quad
  \hat{\bs e}_{\bs x}^0 = \frac{\boldsymbol n_{e_1}}{\boldsymbol n_{e_1} 
  \cdot \boldsymbol t_{e_0}},
   \quad 
  \hat{\bs e}_{\bs x}^1 = \frac{\boldsymbol n_{e_0}}{\boldsymbol n_{e_0} 
  \cdot \boldsymbol t_{e_1}}.
  $$
  \item If $\bs x \in \mathcal X_{\mathring{e}}, e \in \Delta_1(T)$, 
    then
  $$
  \bs e_{\bs x}^0 = \bs t_{e}, \quad \bs e_{\bs x}^1 = \bs n_{e}, \quad	
  \hat{\bs e}_{\bs x}^0 = \bs t_{e}, \quad \hat{\bs e}_{\bs x}^1 = \bs n_{e}.	
  $$  
  \item If $\bs x \in \mathcal X_{\mathring{T}}$, then
  $$
  \bs e_{\bs x}^0 = (1, 0), \quad \bs e_{\bs x}^1 = \bs (0, 1), \quad
  \hat{\bs e}_{\bs x}^0 = (1, 0), \quad \hat{\bs e}_{\bs x}^1 = \bs (0, 1).
  $$		
\end{enumerate}

\begin{figure}[htp]
\centering
\includegraphics[width=0.8\textwidth]{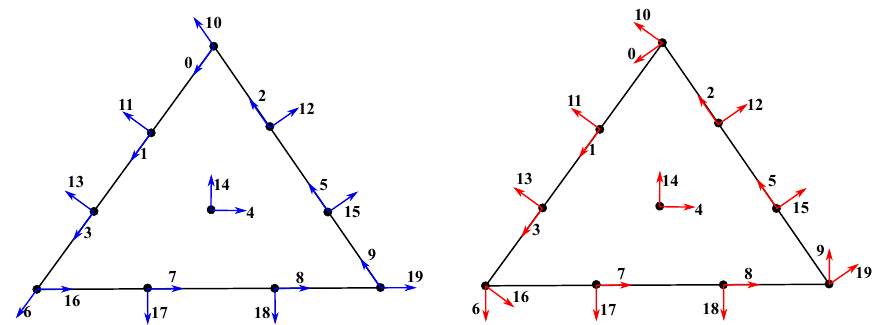}
\caption{The left figure shows $\{\bs e_0, \bs e_1\}$ at each interpolation 
point, the right figure shows $\{\hat{ \bs e}_0, \hat{ \bs e}_1\}$ 
at each interpolation point.}
\end{figure}

\subsubsection{3D basis on tetrahedron meshes}
  Let $T$ be a tetrahedron, for any $\bs x \in \mathcal X_{T}$, 
  define a frame $\{\bs e_{\boldsymbol x}^0, \bs e_{\bs x}^1, \bs e_{\bs x}^2\}$ 
  at $\bs x$ and its dual frame 
  $\{\hat {\bs e}_{\bs x}^0, \hat{\bs e}_{\bs x}^1, \hat{\bs e}_{\bs x}^2\}$ 
  as follows:
\begin{enumerate}
  \item If $\bs x \in \Delta_0(T)$ and adjacent edges of $\bs x$ are 
    $e_0$, $e_1$, $e_2$, adjacent faces of $\bs x$ are 
    $f_0$, $f_1$, $f_2$, then
  $$
  \bs e_{\bs x}^0 = \bs t_{e_0}, \quad 
  \bs e_{\bs x}^1 = \bs t_{e_1}, \quad 
  \bs e_{\bs x}^2 = \bs t_{e_2},
  $$
  $$
  \hat{\bs e}_{\bs x}^0 = \frac{\boldsymbol n_{f_0}}{\boldsymbol n_{f_0} 
    \cdot \boldsymbol t_{e_0}}, \quad 
  \hat{\bs e}_{\bs x}^1 = \frac{\boldsymbol n_{f_1}}{\boldsymbol n_{f_1} 
    \cdot \boldsymbol t_{e_1}}, \quad 
  \hat{\bs e}_{\bs x}^2 = \frac{\boldsymbol n_{f_2}}{\boldsymbol n_{f_2} 
    \cdot \boldsymbol t_{e_2}}.
  $$
  \item If $\bs x \in \mathcal X_{\mathring{e}}, e \in \Delta_1(T)$ 
    and adjacent faces are $f_0$, $f_1$ then
  $$
  \bs e_{\bs x}^0 = \bs t_{e}, \quad 
  \bs e_{\bs x}^1 = \bs n_{f_0}\times \bs t_{e}, \quad 
  \bs e_{\bs x}^2 = \bs n_{f_1}\times \bs t_{e},	
  $$ 
  $$
  \hat{\bs e}_{\bs x}^0 = \bs t_{e}, \quad 
  \hat{\bs e}_{\bs x}^1 = \frac{\boldsymbol n_{f_1}}{\boldsymbol n_{f_1} 
    \cdot (\boldsymbol n_{f_0}\times \boldsymbol t_e)}, \quad	
  \hat{\bs e}_{\bs x}^2 = \frac{\boldsymbol n_{f_0}}{\boldsymbol n_{f_0} 
    \cdot (\boldsymbol n_{f_1}\times \boldsymbol t_e)}.
  $$
  \item If $\bs x \in \mathcal X_{\mathring{f}}, f \in \Delta_2(T)$, 
    the first edge of $f$ is $e$, then
  $$
  \bs e_{\bs x}^0 = \bs t_{e}, \quad 
  \bs e_{\bs x}^1 = \bs t_{e} \times \bs n_f, \quad
  \bs e_{\bs x}^2 = \bs n_{f},
  $$
  $$
  \hat{\bs e}_{\bs x}^0 = \bs t_{e}, \quad 
  \hat{\bs e}_{\bs x}^1 = \bs t_{e} \times \bs n_f, \quad
  \hat{\bs e}_{\bs x}^2 = \bs n_{f}.	
  $$
  \item If $\bs x \in \mathcal X_{\mathring{T}}$, then
  $$
  \bs e_{\bs x}^0 = (1, 0, 0), \quad 
  \bs e_{\bs x}^1 = \bs (0, 1, 0), \quad 
  \bs e_{\bs x}^2 = \bs (0, 0, 1),
  $$
  $$
  \hat{\bs e}_{\bs x}^0 = (1, 0, 0), \quad 
  \hat{\bs e}_{\bs x}^1 = \bs (0, 1, 0), \quad 
  \hat{\bs e}_{\bs x}^2 = \bs (0, 0, 1).
  $$
\end{enumerate}

\begin{figure}[htp]
  \begin{minipage}[t]{0.48\textwidth}
  \centering
  \includegraphics[width=0.5\textwidth]{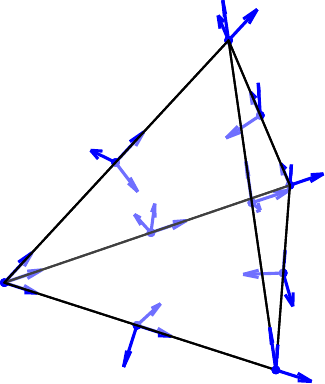}
  \end{minipage}
  \begin{minipage}[t]{0.48\textwidth}
  \centering
  \includegraphics[width=0.5\textwidth]{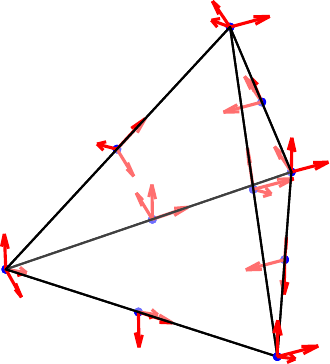}
  \end{minipage}
  \caption{The left figure shows $\{\bs e_0, \bs e_1, \bs e_2\}$ at each 
      interpolation point, the right figure shows 
      $\{\hat{ \bs e}_0, \hat{ \bs e}_1, \hat{\bs e}_2\}$ 
      at each interpolation point.}
\end{figure}

\section{Indexing Management of Degrees of Freedom}\label{sec:implementation}

In Section~\ref{sec:pre}, we have already discussed
the dictionary indexing rule for interpolation points in each element. In this
section, we will address the global indexing rules for Lagrange
interpolation points, ensuring that interpolation points on a sub-simplex,
shared by multiple elements, have a globally unique index.  Given the
one-to-one relationship between interpolation points and DoFs, this is
equivalent to providing an indexing rule for global DoFs in the scalar Lagrange
finite element space. Based on this, we will further discuss the indexing
rules for DoFs in face and edge finite element spaces.

\subsection{Lagrange finite element space}

We begin by discussing the data structure of the tetrahedral mesh, denoted by
\(\mathcal T_h\). Let the numbers of nodes, edges, faces, and cells in
\(\mathcal T_h\) be represented as \lstinline{NN}, \lstinline{NE},
\lstinline{NF}, and \lstinline{NC}, respectively. We utilize two arrays to
represent \(\mathcal T_h\):
\begin{itemize}
  \item \lstinline{node} (shape: \lstinline{(NN,3)}): \lstinline{node[i,j]} represents the $j$-th component of the Cartesian coordinate of the $i$-th vertex.
  \item \lstinline{cell} (shape: \lstinline{(NC,4)}): \lstinline{cell[i,j]} gives the global index of the $j$-th vertex of the $i$-th cell.
\end{itemize}
Given a tetrahedron denoted by \lstinline{[0,1,2,3]}, we define its local edges and faces as:
\begin{itemize}
  \item \lstinline{SEdge = [(0,1),(0,2),(0,3),(1,2),(1,3),(2,3)];}
  \item \lstinline{SFace = [(1,2,3),(0,2,3),(0,1,3),(0,1,2)];}
  \item \lstinline{OFace = [(1,2,3),(0,3,2),(0,1,3),(0,2,1)].}
\end{itemize}
Here, we introduced two types of local faces. The prefix \lstinline{S} implies
sorting, and \lstinline{O} indicates outer normal direction. Both
\lstinline{SFace[i,:]} and \lstinline{OFace[i,:]} represent the face opposite to
the $i$-th vertex but with varied ordering. The normal direction as determined
by the ordering of the three vertices of \lstinline{OFace} matches the outer
normal direction of the tetrahedron. This ensures that the outer normal
direction of a boundary \lstinline{face} points outward from the mesh.
Meanwhile, \lstinline{SFace} aids in determining the global index of the
interpolation points on the face. For an in-depth discourse on indexing,
ordering, and orientation, we direct readers to \mc{sc3} in $i$FEM
\cite{Chen:2008ifem}.

Leveraging the \lstinline{unique} algorithm for arrays, we can derive the
following arrays from \lstinline{cell}, \lstinline{SEdge}, and
\lstinline{OFace}:
\begin{itemize}
  \item \lstinline{edge} (shape: \lstinline{(NE,2)}): \lstinline{edge[i,j]} gives the global index of the $j$-th vertex of the $i$-th edge.
  \item \lstinline{face} (shape: \lstinline{(NF,3)}): \lstinline{face[i,j]} provides the global index of the $j$-th vertex of the $i$-th face.
  \item \lstinline{cell2edge} (shape: \lstinline{(NC,6)}): \lstinline{cell2edge[i,j]} indicates the global index of the $j$-th edge of the $i$-th cell.
  \item \lstinline{cell2face} (shape: \lstinline{(NC,4)}): \lstinline{cell2face[i,j]} signifies the global index of the $j$-th face of the $i$-th cell.
\end{itemize}

Having constructed the edge and face arrays and linked cells to them, we next
establish indexing rules for interpolation points on \(\mathcal T_h\). Let \(k\)
be the degree of the Lagrange finite element space. The number of interpolation
points on each cell is
\[
\mc{ldof} = \dim \mathbb P_k(T) = \frac{(k+1)(k+2)(k+3)}{6},
\]
and the total number of interpolation points on \(\mathcal T_h\) is
\[
\mc{gdof} = \mc{NN} + n_e^k \cdot \mc{NE} + n_f^k \cdot \mc{NF} + n_c^k \cdot \mc{NC},
\]
where
\[
n_e^k = k-1, \quad n_f^k = \frac{(k-2)(k-1)}{2}, \quad n_c^k =
\frac{(k-3)(k-2)(k-1)}{6},
\]
are numbers of interpolation points inside edge, face, and cell, respectively. We need an index mapping from $[0: \mc{ldof}-1]$ to  $[0: \mc{gdof}-1]$. 
See Fig. \ref{fig:assigndofex} for an illustration of the local index and the global index of interpolation points. 


The tetrahedron's four vertices are ordered
according to the right-hand rule, and the interpolation points adhere to the
dictionary ordering map \(R_3(\boldsymbol \alpha)\). As Lagrange
element is globally continuous, the indexing of interpolation points on the boundary $\partial T$ should be global. Namely a unique index for points on vertices, 
edges, faces should be used and a mapping from the local index to the global index is needed.

We first set a global indexing rule for all interpolation points. We sort the index by vertices, edges, faces, and cells. For the
interpolation points that coincide with the vertices, their global index are set
as $0, 1, \ldots, \mc{NN}-1$.  When $k > 1$,  for the interpolation points that
inside edges, their global index are set as 
$$
\small
\begin{matrix}
 0\\
 1\\
 \vdots\\
 \mc{NE}-1
\end{matrix}
\begin{pmatrix}
0\cdot n_e^k &  \cdots & 1\cdot n_e^k-1\\
1\cdot n_e^k &  \cdots & 2\cdot n_e^k-1\\
 \vdots &  & \vdots \\
(\mc{NE}-1)\cdot n_e^k & \cdots & \mc{NE}\cdot n_e^k - 1\\
\end{pmatrix} +  \mc{NN}.
$$
Here recall that $n_e^k = k-1$ is the number of interior interpolation points on an edge. 
When $k>2$, for the interpolation points that inside each face, their global
index are set as  
$$
\small
\begin{matrix}
 0\\
 1\\
 \vdots\\
 \mc{NF}-1
\end{matrix}
\begin{pmatrix}
0\cdot n_f^k &  \cdots & 1\cdot n_f^k-1\\
1\cdot n_f^k &  \cdots & 2\cdot n_f^k-1\\
\vdots &  & \vdots \\
(\mc{NF}-1)\cdot n_f^k & \cdots & \mc{NF}\cdot n_f^k - 1\\
\end{pmatrix} + \mc{NN} + \mc{NE}\cdot n_e^k,
$$
where $n_f^k = (k-2)(k-1)/2$ is the number of interior interpolation points on $f$. 
When $k>3$, the global index of the interpolation points that inside
each cell can be set in a similar way. 

Then we use the two-dimensional array named \lstinline{cell2ipoint} of shape
\lstinline{(NC,ldof)} for the index map. On the $j$-th interpolation
point of the $i$-th cell, we aim to determine its unique global index and store
it in \lstinline{cell2ipoint[i,j]}.


For vertices and cell interpolation points, the mapping is straightforward by
the global indexing rule. Indeed \mc{cell} is the mapping of the local index of a vertex to its global index. 
However, complications arise when the
interpolation point is located within an edge or face due to non-unique representations of an edge and a face. 

We use the more complicated face interpolation points as a typical example to illustrate the situation. 
Consider, for instance, the scenario where the $j$-th interpolation point lies
within the $0$th local face \(F_0\) of the $i$-th cell. Let \(\alpha = \)
\lstinline{m = [m0,m1,m2,m3]} be its lattice point. Given that \(F_0\) is
opposite to vertex $0$, we deduce that \(\lambda_0|_{F_0} = 0\), which implies
\lstinline{m0} is $0$. The remaining components of \lstinline{m} are non-zero,
ensuring that the point is interior to \(F_0\). 

\begin{figure}[htp]
  \begin{minipage}[t]{0.48\textwidth}
  \centering
  \includegraphics[width=0.82\textwidth]{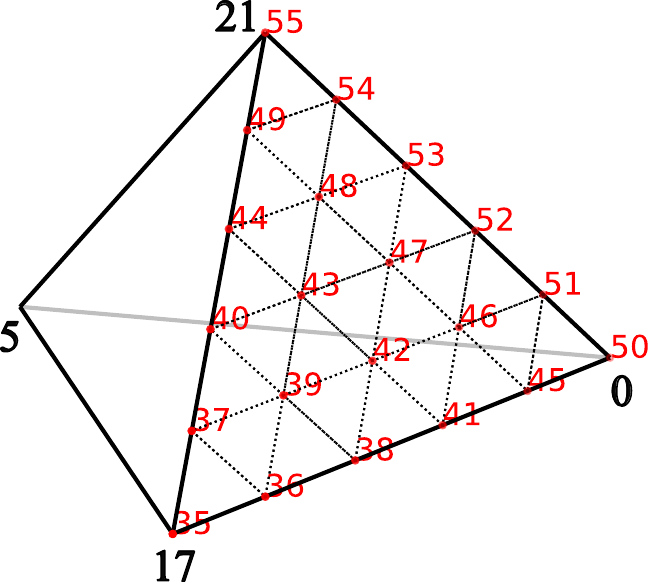}
  \end{minipage}
  \begin{minipage}[t]{0.48\textwidth}
  \centering
  \includegraphics[width=0.6\textwidth]{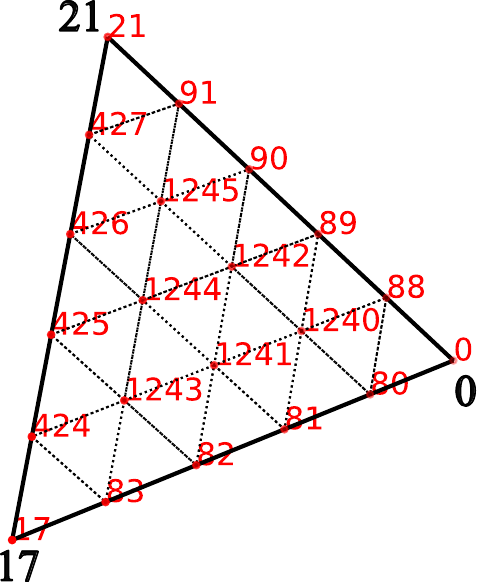}
  \end{minipage}
  \caption{Local indexing (Left) and global indexing(Right) of DoFs for a face
  of tetrahedron, where the local vertex order is \mc{[17, 0,
  21]} and global vertex order is \mc{[0, 17, 21]}. Due to the different ordering of vertices in local and global representation of the face, the ordering of the local indexing and the global indexing is different.}
  \label{fig:assigndofex}
\end{figure}

Two representations for the face with global index \lstinline{cell2face[i,0]} are subsequently acquired:
\begin{itemize}
  \item \lstinline{LFace = cell[i,SFace[0,:]]} (local representation)
  \item \lstinline{GFace = face[cell2face[i,0],:]} (global representation)
\end{itemize}
Although \lstinline{LFace} and \lstinline{GFace} comprise identical vertex
numbers, their ordering differs. For example, \mc{LFace = [5 6 10]} while \mc{GFace =[10 6 5]}. 

The array \lstinline{m = [m1,m2,m3]} has a
one-to-one correspondence with the vertices of \lstinline{LFace}. To match this
array with the vertices of \lstinline{GFace}, a reordering based on argument
sorting is performed:
\begin{lstlisting}
i0 = argsort(argsort(GFace));
i1 = argsort(LFace);
i2 = i1[i0];
m = m[i2];
\end{lstlisting}
In the example of \mc{LFace = [5 6 10]} while \mc{GFace = [10 6 5]}, the input \lstinline{m = [m1,m2,m3]} will be reordered to \lstinline{m = [m3,m2,m1]}.

From the reordered \lstinline{m = [m1,m2,m3]}, the local index \(\ell\) of the
$j$-th interpolation point on the global face \lstinline{f = cell2face[i,0]} can
be deduced:
\begin{align*}
   \ell = & \frac{(\mc{m2}-1 + \mc{m3}-1)(\mc{m2}-1 + \mc{m3}-1 + 1)}{2} + \mc{m3}-1 \\
   =& \frac{(\mc{m2} + \mc{m3} -2)(\mc{m2} + \mc{m3} - 1)}{2} + \mc{m3}-1.
\end{align*}
It's worth noting that the index of interpolation points solely within the face
needs consideration. Finally, the global index for the $j$-th
interpolation point within the $0$th local face of the $i$-th cell is:
\[
    \mc{cell2ipoint[i j]} = \mc{NN} + n_e^k\cdot \mc{NE} + n_f^k\cdot \mc{f} + \ell.
\]

Here, we provide a specific example. Consider a $5$th-degree Lagrange finite
element space on the \mc{c}-th tetrahedron in a
mesh depicted in Fig. \ref{fig:assigndofex}. The vertices of this tetrahedron
are \mc{[5,17,0,21]}, and its 0th face is denoted as \mc{f}, with vertices 
\mc{[0,17,21]}. Assuming that 
\[
    \mc{NN} + n_e^k\cdot \mc{NE} + n_f^k\cdot \mc{f} = 1240.
\]
For the $39$th local DoF on tetrahedron, 
its corresponding multi-index is \mc{[0,3,1,1]} on cell and  \mc{[3,1,1]} 
on face. Since the local face is \mc{[17,0,21]} and the global face is
\mc{[0,17,21]}, then \mc{m=[1,3,1]} and $\ell$ \mc{ =3} thus
\[
    \mc{cell2ipoint[c,39] = 1243}.
\]
Similar, for the $43$th local DoF on tetrahedron, \mc{m = [1,2,2]} and
\mc{$\ell$=4} thus
\[
    \mc{cell2ipoint[c,43] = 1244}.
\]
Fig. \ref{fig:assigndofex} shows the correspondence between the local and the global indexing of DoFs for the cell.

In conclusion, we have elucidated the construction of global indexing for
interpolation points inside cell faces. This method can be generalized for edges
and, more broadly, for interior interpolation points of the low-dimensional
sub-simplex of an \(n\)-dimensional simplex. Please note that for scalar
Lagrange finite element spaces, the \mc{cell2ipoint} array is the mapping array
from local DoFs to global DoFs.

\subsection{Face  and edge finite element spaces}

First, we want to emphasize that the management of DoFs is to manage the
continuity of finite element space.  

The face element space and edge element space are vector finite element spaces, which define DoFs of
vector type by defining a vector frame on each interpolation point.
At the same time, they define their vector basis functions by combining the Lagrange
basis function and the dual frame of the DoFs.

Alternatively, we can say each DoF in the face element or edge element corresponds to a unique interpolation point
$p$ and a unique vector $\boldsymbol e$, and each basis funcion also corresponds
to a unique Lagrange basis function which is defined on $p$ and a vector
$\hat{\boldsymbol e}$ which is the dual to  $\boldsymbol e$. 

The management of DoFs is essentially a counting problem. First of all, we need
to set global and local indexing rules for all DoFs. 
We can globally divide the DoFs into shared and unshared among simplexes. The
DoFs shared among simplexes can be further divided into on-edge and on-face
according to the dimension of the sub-simplex where the DoFs locate. First count the shared DoFs on
each edge according to the order of the edges, then count the shared DoFs on
each face according to the order of the faces, and finally count the unshared
DoFs in the cell. On each edge or face, the DoFs' order can follow the order of
the interpolation points. Note
that, for the face element space and edge element space there are no DoFs shared on nodes. And for the
3D face element space there are no DoFs shared on edges. So the global numbering rule is
similar with the Lagrange interpolation points. 

According to the global indexing rule, we also can get an array named
\lstinline{dof2vector} with shape \lstinline{(gdof, GD)}, where \lstinline{gdof}
is the number of global DoFs and \lstinline{GD} represent geometry dimensions.
And \lstinline{dof2vector[i, :]} store the vector of the $i$-th DoF.

Next we need to set a local indexing rules and generate a array
\lstinline{cell2dof} with shape \lstinline{(NC,ldof)}, where \lstinline{ldof}
is the number of local DoFs on each cell. Note that each DoF was determined by an
interpolation point and a vector. And for each interpolation point, there is a frame
(including \lstinline{GD} vectors) on it. Given a DoF on the $i$-th cell, denote the
local index  of its interpolation point as $p$, and the local index 
of its vector in the frame denote as $q$, then one can set a unique local index
number $j$ by $p$ and $q$, for example
$j = n\cdot q + p$,
where $n$ is the number of interpolation points in the $i$-th cell.  Furthermore, 
we can compute the \lstinline{cell2dof[i,j]} by the global index 
\lstinline{cell2ipoint[i, p]}, the sub-simplex that the interpolation point
locate, and the global indexing rule. 

\begin{remark}\rm 
  Note that the local and global number rules  mentioned above are not unique.
  Furthermore, with the array \lstinline{cell2dof}, the implementation of these higher-order
  finite element methods mentioned in this paper is not fundamentally different
  from the conventional finite element in terms of matrix vector assembly and
  boundary condition handling. $\qed$
\end{remark}

\section{Numerical Examples}\label{sec:numerexamples}

In this section, we numerically verify the 3-dimensional face element basis and edge element basis using two test problems over the domain $\Omega =
(0, 1)^3$ partitioned into a structured tetrahedron mesh $\mathcal T_h$. All the algorithms and numerical examples are implemented based on
the FEALPy package~\cite{FEALPy}.


\subsection{High Order Elements for Poisson Equation in the Mixed Formulation}
Consider the Poisson problem:
$$
\begin{aligned}
  \left\{
    \begin{aligned}
      \boldsymbol u + \nabla p ={}& 0 \quad \text{ in } \ \Omega,\\
      \quad \mathrm{div}\,\boldsymbol u ={}& f\quad \text{ in } \ \Omega,\\
      \quad \quad \ \ \  p ={}& g \quad \text{ on } \ \partial \Omega.
    \end{aligned}
  \right.
\end{aligned}
$$
The variational problem is : find 
$\boldsymbol u \in H(\mathrm{div}, \Omega)$, $p \in L^2(\Omega)$ to satisfy:

\begin{align*}
  \int_{\Omega} \boldsymbol u \cdot \boldsymbol v \ \mathrm d\boldsymbol x -
  \int_{\Omega} p\,\mathrm{div}\, \boldsymbol v \ \mathrm d\boldsymbol x & =
  -\int_{\partial \Omega} g(\boldsymbol v \cdot \boldsymbol n) \
  \mathrm d\boldsymbol s, \quad \forall~\boldsymbol v \in H(\mathrm{div}, \Omega),\\
  -\int_{\Omega} (\mathrm{div}\, \boldsymbol u) q \ \mathrm d\boldsymbol x & =
  -\int_{\Omega} fq \ \mathrm d\boldsymbol x, \qquad\quad\;\;  \forall~q \in L^2(\Omega).
\end{align*}

Let $V_{h,k}^{\div}$ be the BDM space with degree $k$ on $\mathcal T_h$ and
piecewise polynomial space of degree $k-1$ on $\mathcal T_h$ by $V_{k-1}$.  
The corresponding finite element method is: find 
$\boldsymbol u_h \in V_{h,k}^{\div}, p_h \in V_{k-1}$, such that
\begin{align}
  \int_{\Omega} \boldsymbol u_h \cdot \boldsymbol v_h \ \mathrm d\boldsymbol x -
  \int_{\Omega} p\,\mathrm{div}\,\boldsymbol v_h \ \mathrm d\boldsymbol x & =
  -\int_{\partial \Omega} g(\boldsymbol v_h \cdot \boldsymbol n) \
  \mathrm d\boldsymbol s \quad  \forall~\boldsymbol v_h \in V_{h,k}^{\div},\label{exm2eq1}\\
  -\int_{\Omega} (\mathrm{div}\,\boldsymbol u_h) q_h \ \mathrm d\boldsymbol x &=
  -\int_{\Omega} fq_h \ \mathrm d\boldsymbol x, \qquad\quad\;  \forall~q_h \in V_{k-1}.\label{exm2eq2}
\end{align}

\begin{figure}[htp]
\centering
\includegraphics[width=0.6\textwidth]{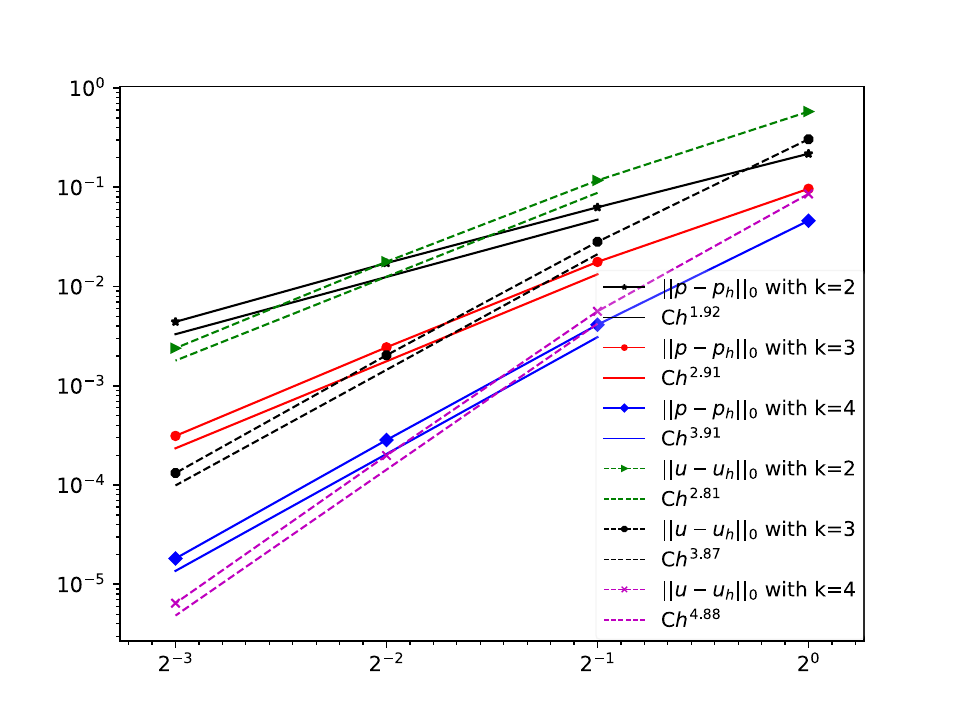}
\caption{Errors $\|\bs u - \bs u_h\|_0$ and $\|p - p_h\|_0$ of finite element
method \eqref{exm2eq1} and \eqref{exm2eq2} on uniformly refined mesh with $k = 2, 3, 4$.}
\label{exm2fig1}
\end{figure}

To test the convergence order of BDM space, we set
$$
\boldsymbol u = 
\begin{pmatrix}
-\pi\sin(\pi x)\cos(\pi y)\cos(\pi z)\\ 
-\pi\cos(\pi x)\sin(\pi y)\cos(\pi z)\\
-\pi\cos(\pi x)\cos(\pi y)\sin(\pi z)
\end{pmatrix},
$$
\begin{align*}
& p = \cos(\pi x)\cos(\pi y)\cos(\pi z), &  & f = 3\pi^2\cos(\pi x)\cos(\pi y)\cos(\pi z).&
\end{align*}

The numerical results are shown in Fig.~\ref{exm2fig1} and it is clear that 
$$
\|\bs u-\bs u_h\|_{0} \leq C h^{k+1}, \quad \|p - p_h\|_{0} \leq Ch^{k}.
$$

\subsection{High Order Elements for Maxwell Equations}
Consider the time harmonic Maxwell equation:
$$
\begin{cases}
\nabla \times (\mu^{-1}\nabla \times \boldsymbol E) - \omega^2 \tilde \epsilon
\boldsymbol E & = \boldsymbol J \quad \text{ in } \ \ \Omega,\\
\qquad\qquad\qquad\qquad \boldsymbol n \times \boldsymbol E & = \boldsymbol 0 \quad \text{ on } \ \
\partial \Omega,
\end{cases}
$$
where $\mu$ is the relative permeability, $\tilde \epsilon$ is the relative permittivity, and $\omega$ is the wavenumber.
The variational problem is: find $\boldsymbol E \in H_0(\mathrm{curl}, \Omega)$ to satisfy:
\begin{align*}
\int_{\Omega}\mu^{-1}(\nabla \times \boldsymbol E)
\cdot (\nabla \times \boldsymbol v) \ \mathrm d\boldsymbol x - 
\int_{\Omega}\omega^2 \tilde 
\epsilon \boldsymbol E \cdot \boldsymbol v\ \mathrm d\boldsymbol x
= \int_{\Omega}\boldsymbol J \cdot \boldsymbol v \ \mathrm d\boldsymbol x, \quad   \forall~\boldsymbol v \in H_0(\curl, \Omega).
\end{align*}
Let $\mathring{V}_{h,k}^{\curl} = V_{h,k}^{\curl} \cap H_0(\curl, \Omega)$, where 
$V_{h,k}^{\curl} $ is the edge element space defined in Section~\ref{sec:nedelecbasis}. The corresponding finite element method is: find 
$\boldsymbol E_h \in \mathring{V}_{h,k}^{\curl}$ s.t.
\begin{align}
\label{exm1:eq1}
\int_{\Omega}\mu^{-1}(\nabla \times \boldsymbol E_h)
\cdot (\nabla \times \boldsymbol v_h) \ \mathrm d\boldsymbol x -
\int_{\Omega}\omega^2 \tilde
\epsilon \boldsymbol E_h \cdot \boldsymbol v_h\ \mathrm d\boldsymbol x
= \int_{\Omega}\boldsymbol J \cdot \boldsymbol v_h \ \mathrm d\boldsymbol x, \quad \forall~
\boldsymbol v_h \in  \mathring{V}_{h,k}^{\curl}.
\end{align} 

\begin{figure}[htp] 
\centering
\includegraphics[width=0.6\textwidth]{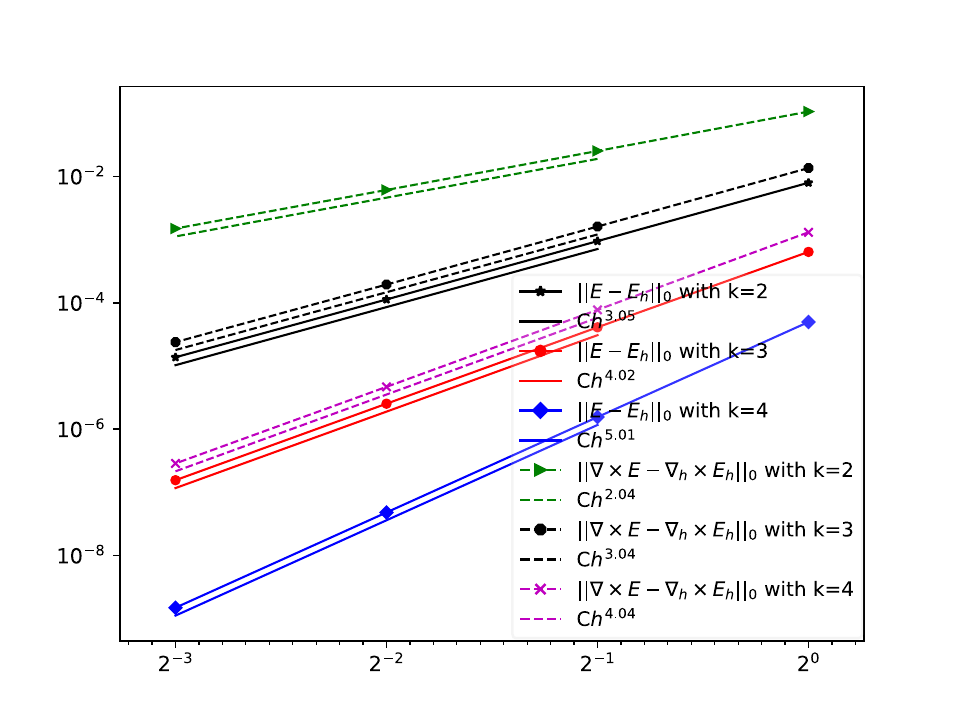}
\caption{Errors $\|\bs E - \bs E_h\|_{0}$ and 
    $\|\nabla \times(\bs E - \bs E_h)\|_{0}$ of finite element
    method \eqref{exm1:eq1} on uniformly refined mesh with $k = 2, 3, 4$.}
\label{exm1fig1}
\end{figure}

To test the convergence rate of second-kind N\'ed\'elec space, we choose 
\begin{align*}
&\omega  = \tilde \epsilon = \mu = 1, & \quad & \boldsymbol E = (f, \sin(x)f, \sin(y)f),&\\
& f = (x^2-x)(y^2-y)(z^2-z),& \quad & \boldsymbol J  = \nabla\times\nabla\times \boldsymbol E - \boldsymbol E. &
\end{align*}

The numerical results are shown in Fig.~\ref{exm1fig1} and it is clear that 
$$
\|\bs E-\bs E_h\|_{0} \leq C h^{k+1}, \quad \|\nabla\times(\bs E-\bs E_h)\|_{0} \leq C h^k.
$$

\bibliographystyle{abbrv}
\bibliography{paper}
\end{document}